\documentclass[10pt,draft]{article}
\usepackage{mathrsfs}
\usepackage{amsmath, bm, psfrag, epsfig}
\usepackage{latexsym}
\usepackage{amssymb}

 \makeatother
\usepackage{color}
\usepackage{cite}
\usepackage{indentfirst}
\usepackage{geometry}
\usepackage{epsfig}
\usepackage{epstopdf}
\usepackage{float}

\usepackage{enumerate}
\usepackage{hhline}

\usepackage{lineno}

\usepackage{graphicx}
\newtheorem{thm}{Theorem}[section]

\newtheorem{lemma}[thm]{Lemma}

\newtheorem{defn}[thm]{Definition}
\newtheorem{cor}[thm]{Corollary}
\newtheorem{ques}[thm]{Question}
\newtheorem{prop}[thm]{Proposition}
\newtheorem{rem}{Remark}[section]

\newtheorem{ex}{Example}[section]
\newtheorem{Notation}[thm]{Notation}

\newenvironment{proof}{{\bf Proof}.}{\rule{2.5mm}{2.5mm}}

\begin{document}

\title{\bf A note on knowledge structures delineated by fuzzy skill multimaps\footnote{This work is supported by the Key Program of the Natural Science Foundation of Fujian Province (No: 2020J02043), the National Natural Science Foundation of China (Grant No. 11571158) and Natural Science Foundation of Fujian Province (No.2018J01423).}}
\author{Xiyan Cao~$^1$, \, Fucai Lin~$^{1, 3}$\,\thanks{Email addresses:linfucai@mnnu.edu.cn}, Wen Sun$^2$, Jinjin Li~$^1$\,\\
$^1$ School of Mathematics and Statistics,\\ Minnan Normal University, Zhangzhou 363000, China\\
 $^2$ Department of mathematics, Shantou University, Shantou 515063, China\\
 $^3$ Fujian Key Laboratory of Granular Computing and Applications,\\
Minnan Normal University, Zhangzhou 363000, P. R. China}

\date{}

\maketitle

\begin{center}
\begin{minipage}{140mm}
{\bf Abstract:} {Fuzzy skill multimaps can describe individuals' knowledge states from the perspective of latent cognitive abilities. The significance of discriminative knowledge structure is reducing repeated testing and the workload for students, which responds to the so-called `double reduction policy'. As an application in knowledge assessment, the bi-discriminative knowledge structure is useful for evaluating `relative independence' items or knowledge points. Moreover, it is of great significance that which kind of fuzzy skill multimap delineates a discriminative or bi-discriminative knowledge structure. Moreover, we give a characterization of fuzzy skill multimaps such that the delineated knowledge structures are knowledge spaces, learning spaces and simple closure spaces respectively. Further, there are some interesting applications in the meshing of the delineated knowledge structures and the distributed fuzzy skill multimaps. We also give a more precise analysis with regard to the separability (resp. bi-separability) on which condition the information collected within a local assessment can reflect an assessment at the global level, and which condition any of the local domains can localize from given global assessment.}

\vspace{1mm}

\end{minipage}
\end{center}
\section{Introduction}
Doignon and Falmagne (1999) introduced the theory of knowledge spaces (KST) which is regarded as a mathematical framework in order to assess the knowledge and suggest for further learning~\cite{falmagne2011learning,ref2,ref7}. In this paper, a field of {\it knowledge} is a non-empty set of items or questions, denoted by $Q$. A subset $H$ of $Q$ is called a {\it knowledge state} whether an individual is capable of solving it in ideal conditions. A family of knowledge states is called a {\it knowledge structure} if it contains~$\emptyset$ and~$Q$~\cite{falmagne2011learning,ref2}.

In KST, there are some method to for construct the knowledge state of an individual and an accurate knowledge structure, such as query-routine~\cite{ref8,ref9,ref10,ref11,ref12,ref13,ref14,ref15} and ps-query routine~\cite{ref16}, which are designed to establish knowledge states and knowledge structures by some experienced experts. Moreover, by analyzing the relationship between items and skills is another effective method to construct knowledge ~\cite{ref9,ref17,ref18,ref19,ref20,ref21}. If $S$ is a finite skills consisting in method, tricks or algorithms etc, then the relationship can be builded by a skill map or a skill multimap.

In real-life case, each individual may have different level of proficiency in skills. In \cite{ref5}, the authors represented the level of proficiency in skills by a fuzzy set of skills. Indeed, they say that a fuzzy skill map~$(Q, S, \tau)$ and a fuzzy skill multimap~$(Q, S, \mu)$ assign each item~$q$ a fuzzy set~$\tau(q)$ and a family~$\mu(q)$ of fuzzy sets of~$S$ respectively, where each fuzzy set~$C$ in~$\mu(q)$ is called a competency. Similar to skill map, a knowledge structure can also be delineated by fuzzy skill maps or fuzzy skill multimaps. In \cite{ref5}, there are some algorithms in order to delineate a knowledge structure via the conjunctive model, the disjunctive model and the competency model.

Knowledge space is an important knowledge structure~\cite{falmagne2011learning}, which implies any two distinct individuals can get the same progressive knowledge state after fully communicating with each other. That is, they can get knowledge state of~$K\bigcup L$ if one state is~$K$ and another is~$L$.

As in the equal importance of knowledge space, simple closure space can be used for diagnosing items individuals can not solve and shrinking the items they may not solve after communication. Indeed, after communication, the items they can not solve become~$K^{c} \bigcap L^{c}$ if the items one can not solve are~$K^{c}$ and another are~$L^{c}$.

As a special knowledge space, learning space can guide individuals following the learning path step by step~\cite{falmagne2011learning}. Moreover, the learning path of learning space is well-graded~\cite{falmagne2011learning}, hence there exists one-by-one steps from~$\emptyset$ to~$Q$.

In summary, fuzzy skill multimap is an important approach which builds up the relationship between fuzzy skills and items. If we assign a family of fuzzy skills to each item, then the fuzzy skill
multimap goes from performance level to competence level. It follows from the algorithms designed in \cite{ref5} that the knowledge structure can be delineated by fuzzy skill multimaps via the competence model. However, there exists a delineated knowledge structure such that it is not a simple closure space or a knowledge space. In \cite[Problem 6 in Chapter 6 ]{falmagne2011learning}, the authors asked readers under which condition on a skill multimap such that the delineated structure is a knowledge space. In general, it is meaningful and valuable to characterize fuzzy skill multimaps such that the delineated knowledge structures are knowledge spaces, simple closure spaces and learning spaces respectively.

Therefore, it is interesting to consider the question: which kind of fuzzy skill multimap can delineate simple closure space, knowledge space or learning space respectively, that is, the following Question~\ref{1-5}.

\begin{ques}\label{1-5}
Under which condition such that the knowledge structures delineated by fuzzy skill multimaps are knowledge space, simple closure space and learning space respectively?
\end{ques}

The discriminative knowledge structure and separability of items are both important basic concepts in KST~\cite{falmagne2011learning}. By introduced in their book~\cite{falmagne2011learning}, the notions may be useful for manufacturing a discriminative knowledge structure~$(Q^{*},\mathscr{K}^{*})$ from a knowledge structure~$(Q,\mathscr{K})$.

\begin{ex}\label{ddd}
Let~$(Q,\mathscr{K})$ be a knowledge structure, where $Q=\{a, b, c, d\}$ and
\begin{equation*}
\begin{aligned}
&\mathscr{K}=\{\emptyset, \{d\}, \{a, c\}, \{a, b, c\}, \{a, d, c\}, Q\}.
\end{aligned}
\end{equation*}
\end{ex}
Then we have
$$\mathscr{K}_{a}=\mathscr{K}_{c}=\{\{a, c\}, \{a, b, c\}, \{a, d, c\}, Q\}.$$
In other word, items~$a$ and~$c$ carry the same information for $\mathscr{K}$. Hence, only one of the two items~$a$ and~$c$ needs to be asked in the test of acquired knowledge of a subject. Since~$\mathscr{K}_{a}=\mathscr{K}_{c}$, knowledge structure~$(Q,\mathscr{K})$ is not discriminative by Definition~\ref{ee}(1). Clearly, we have three notions as follows:
\begin{equation*}
\begin{aligned}
&a^{*}=\{a, c\}, b^{*}=\{b\}, d^{*}=\{d\}.
\end{aligned}
\end{equation*}
Put~$Q^{*}=\{a^{*}, b^{*}, d^{*}\}$. Then we can construct a discriminative knowledge structure
\begin{equation*}
\begin{aligned}
&\mathscr{K}^{*}=\{\emptyset, \{a^{*}\}, \{d^{*}\}, \{a^{*}, b^{*}\}, \{a^{*}, d^{*}\}, Q^{*}\},
\end{aligned}
\end{equation*}
which is induced by~$\mathscr{K}$ on~$Q^{*}$.

It follows that the implication of discriminative knowledge structure is not only to reduce repeated testing, but also to reduce the workload for students, which responds to the so-called `double reduction policy'. In order to continue to standardize off-campus training and improve the level of education in schools, `double reduction policy' aims to effectively reduce the burden of off-campus training and heavy homework for students in the compulsory education, which thoroughly implemented the fifth Plenary Session of the 19th Central Committee of the Communist Party of China and the spirit of the 19th National Congress of the Communist Party of China.

Take arbitrary two items~$q, r\in Q$. The separability of~$q$ and~$r$ means there exists a state which contains item~$r$ without item~$q$, or a state which contains item~$q$ without item~$r$. In other word, the item~$q$ is not a prerequisite of item~$r$ or the item~$r$ is not a prerequisite of item~$q$. The bi-separability of~$q$ and~$r$ means there exists a state which contains item~$r$ without item~$q$, and a state which contains item~$q$ without item~$r$. In other word, there is no prerequisite relationship between~$q$ and~$r$. It can be called bi-discriminative knowledge structure if each item is `independent' from others~\cite{ref4}.

In~\cite{ref4}, Ge and Li have discussed the separability and bi-separability in knowledge structures which are delineated by skill multimaps respectively. In their paper, the concept of bi-separability is regarded as a bidirectional separability in KST. Therefore, it is natural to pose a question: What about the relationship between fuzzy skill multimaps~$(Q, S, \mu)$ and the separability and bi-separabilit) in knowledge structures which are delineated by fuzzy skill multimaps respectively? A question is posed as follows.

\begin{ques}~\label{kkk}
 Are there sufficient and necessary conditions such that the knowledge structures, which are delineated by fuzzy skill multimaps, are discriminative and bi-discriminative respectively. And what are they?
\end{ques}

In~\cite{ref22}, Heller and Repitsch posed the following question in the conclusion section.

\begin{ques}~\label{kk}
(1) Can an assessment at the global level build upon the information collected within a local assessment?

(2) Is it possible to localize a given global assessment to any of the local domains?
\end{ques}

In \cite{ref4}, Ge and Li have discussed Question~\ref{kk} from the separability and bi-separability in knowledge structures which are delineated by skill multimaps respectively. In this paper, we shall discuss Question~\ref{kk} from the separability and bi-separability in knowledge structures which are delineated by fuzzy skill multimaps respectively. In other words, around the meshing of the delineated knowledge structures and distributed fuzzy skill multimaps, we aim to find some conditions such that substructures can inherit the separability and bi-separability from parent knowledge structure, and some conditions such that parent knowledge structure can preserve the separability and bi-separability from substructures respectively.

The paper is organized as follows: section 2 outlines some concepts of KST. The discussion about Questions~\ref{1-5} and ~\ref{kkk} is demonstrated in section 3.1 and section 3.2 respectively. Around the meshing of the delineated knowledge structure and distributed fuzzy skill multimaps, Question~\ref{kk} is discussed in section 3.3. Finally, section 4 summarizes the main results of our research.

\section{Preliminaries}

Some basic concepts and fuzzy sets operators are defined in this section.

\begin{defn}[\cite{ref2}]
{\rm A {\it knowledge structure} is a pair~$(Q,\mathscr{K})$, where $Q$ is a nonempty set and $\mathscr{K}$ is a family of subsets of~$Q$ such that $\mathscr{K}$ contains at least the empty set~$\emptyset$ and~$Q$. We say that~$Q$ is the {\it domain} of the knowledge structure, each element of $Q$ is referred to as {\it items or questions} and each element of $\mathscr{K}$ is labeled {\it knowledge state}. Moreover, we denote the subset $\{K\in \mathscr{K}|q\in K\}$ of $\mathscr{K}$ for any $q\in Q$ by~$\mathscr{K}_{q}$. In particular, if a knowledge structure~$(Q,\mathscr{K})$ is closed under union (resp. intersection), then $\mathscr{K}$ is called a {\it knowledge space} (resp. {\it simple closure space}). Further, a knowledge structure~$(Q,\mathscr{K})$ is called a {\it quasi ordinal space} if it is both a knowledge space and a simple closure space.}
\end{defn}

\begin{defn}[~\cite{falmagne2011learning}]
{\rm A family of sets $\mathscr{F}$ is {\it well-graded} if, for any pair of distinct sets
$K, L\in\mathscr{F}$, there are a finite sequence of states $K=K_{0}, K_{1}, \ldots, K_{p}=L$, where $p=d(K, L)$ and $d(K_{i-1}, K_{i})=1$ for $1\leq i\leq p$.}
\end{defn}

\begin{defn}[~\cite{falmagne2011learning}]
A knowledge structure~$(Q,\mathscr{K})$ is called a learning space if it is a well-graded knowledge space.
\end{defn}

\begin{defn}~\label{ee}
{\rm Assume $(Q, \mathscr{K})$ is a knowledge structure.

\smallskip
(1)~$(Q, \mathscr{K})$ is said to be {\it discriminative}~\cite{ref3} if~$\mathscr{K}_{q}\neq \mathscr{K}_{r}$ for any pair of distinct items~$q, r\in Q$.

\smallskip
(2)~$(Q, \mathscr{K})$ is said to be {\it bi-discriminative}~\cite{ref4} if~$\mathscr{K}_{q}\not\subseteq \mathscr{K}_{r}$ and~ $\mathscr{K}_{r}\not\subseteq \mathscr{K}_{q}$ for any two distinct items $q, r\in Q$.}
\end{defn}

From Definition~\ref{ee}, bi-discriminative knowledge structure is discriminative. However, discriminative knowledge structure is not necessary bi-discriminative. Similar to the introduction of topological spaces, the separability and bi-separability of knowledge structures can be characterized in Definition~\ref{eee} respectively.

\begin{defn}[\cite{ref4}]~\label{eee}
{\rm Assume $(Q,\mathscr{K})$ is a knowledge structure.

\smallskip
(1) We say that $(Q, \mathscr{K})$ is a {\it $T_{0}$-knowledge structure} if for any two distinct items~$q_{1}, q_{2}\in Q$, there exists a state~$K\in \mathscr{K}$ such that $K\cap\{q_{1}, q_{2}\}$ is just an one-point set.

(2)We say that $(Q, \mathscr{K})$ is a {\it $T_{1}$-knowledge structure} if for any two distinct items~$q_{1}, q_{2}\in Q$, there are states~$K, L\in \mathscr{K}$ such that~$K\cap\{q_{1}, q_{2}\}=\{q_{1}\}$ and~$L\cap\{q_{1}, q_{2}\}=\{q_{2}\}$.}
\end{defn}

\begin{lemma}[\cite{ref4}]\label{1-1}
Assume~$(Q,\mathscr{K})$ is a knowledge structure, then the following (1) and (2) hold.

\smallskip
(1)~$(Q, \mathscr{K})$ is a $T_{0}$-knowledge structure iff~$(Q, \mathscr{K})$ is discriminative.

\smallskip
(2)~$(Q, \mathscr{K})$ is a $T_{1}$-knowledge structure iff $(Q, \mathscr{K})$ is bi-discriminative.
\end{lemma}

It is well known that subspaces of any topological space can inherit~$T_{0}$-separation and $T_{1}$-separation respectively. Similarly, substructures of knowledge structure can inherit separability (resp. bi-separability) from knowledge structure.

\begin{defn}[\cite{ref7}]
{\rm Assume that $(Q, \mathscr{K})$ is a knowledge structure and~$Q'$ is a nonempty subset of~$Q$. Put $\mathscr{K}'=\{K'|K'=K\cap Q', K\in \mathscr{K}\}$. Then~$(Q', \mathscr{K}')$ is a knowledge structure, which is said to be a {\it substructure} of parent structure~$(Q, \mathscr{K})$. Moreover, we say that~$\mathscr{K}'$ is the {\it trace of~$\mathscr{K}$ on~$Q'$}, and is denoted by~$\mathscr{K}|_{Q'}$.}
\end{defn}

\begin{prop}[\cite{ref4}]\label{dd}
Assume~$(Q', \mathscr{K}')$ is a substructure of the knowledge structure~$(Q,\mathscr{K})$. If $(Q,\mathscr{K})$ is discriminative (resp. bi-discriminative), then $(Q',\mathscr{K}')$ is discriminative (resp. bi-discriminative).
\end{prop}

For an arbitrary non-empty set~$S$ of skills, let~$\mathscr{F}(S)$ be the set as follows: $$\mathscr{F}(S)=\{T|T:S\rightarrow [0,1]\}.$$ For each $T\in\mathscr{F}(S)$, we denote~$T$ by~$T=\{\frac{T(s)}{s}: s\in S\}\in \mathscr{F}(S)$, or simply denoted by $T=\{\frac{T(s)}{s}\}$ if the basic set~$S$ is clear from the context. Moreover, we omit $\frac{T(s)}{s}$ if $T(s)=0$ for any $s\in S$. For each~$T\in \mathscr{F}(S)$ and $s\in S$,~$T(s)$ is represented the degree of membership of~$s$ with respect to~$T$. Some operations on~$\mathscr{F}(S)$ are defined as follows~\cite{ref6}:
\begin{equation*}
\begin{aligned}
&T_{1}=T_{2} \Leftrightarrow T_{1}(s)=T_{2}(s) ,\,\, \forall s\in S;\\
&T_{1}\subseteq T_{2} \Leftrightarrow T_{1}(s)\leq T_{2}(s) ,\,\, \forall s\in S;\\
&(T_{1}\cup T_{2})(s)=T_{1}(s)\vee T_{2}(s),\,\, \forall s\in S;\\
&(T_{1}\cap T_{2})(s)=T_{1}(s)\wedge T_{2}(s),\,\, \forall s\in S.
\end{aligned}
\end{equation*}

\section{Main results}
\subsection{The knowledge structures delineated by fuzzy skill multimaps}

It is well known that different levels of proficiency in skills need to solve different items. Therefore, we can build the relationship between items and fuzzy skills by a fuzzy skill multimap.

\begin{defn}[\cite{ref5}]
{\rm A {\it fuzzy skill multimap} is a triple~$(Q, S, \mu)$, where~$Q=\{q_{i}\}_{i\in I}$ is a non-empty set of items, $S=\{s_{j}\}_{j\in J}$ is a non-empty set of skills and~$\mu$ is a mapping from~$Q$ to~$2^{\mathscr{F}(S)\backslash \emptyset}\backslash \emptyset$ such that $\mu(q)$ is a non-empty subset of $\mathscr{F}(S)$ for each~$q\in Q$,.}
\end{defn}

Each element of~$\mu(q)$ is called {\it competencies} for each~$q\in Q$. Moreover, if the elements of each $\mu(q)$ are pairwise incomparable, then~$(Q, S, \mu)$ is said to be a {\it fuzzy skill function}.

\begin{defn}\label{aaaa}{\rm
Assume~$(Q, S, \mu)$ is a fuzzy skill multimap. We say that a mapping~$p: \mathscr{F}(S)\rightarrow 2^{Q}$ is called the {\it problem function induced by~$(Q, S, \mu)$} if $p$ is defined by
$$p(F)=\{q\in Q|\ \mbox{there eixsts}~C_{q}\in \mu(q)\ \mbox{such that}\ ~C_{q}\subseteq F\}$$
for each ~$F\in\mathscr{F}(S)$.}
\end{defn}

\begin{rem}
{\rm For the problem function $p$ induced by~$(Q, S, \mu)$, put~ $\mathscr{K}=\{p(F)|F\in \mathscr{F}(S)\}$. Clearly, the state~$Q$ is delineated by~$T=\{\frac{1}{s}: s\in S\}$ and the state~$\emptyset$ is delineated by $T=\{\frac{0}{s}: s\in S\}$. Therefore, $(Q, \mathscr{K})$ is a knowledge structure, and we say that $(Q, \mathscr{K})$ is the {\it delineated knowledge structure} by a fuzzy skill multimap $(Q, S, \mu)$. This model is called {\it the competency model}. It is easily checked that the knowledge structure delineated by a fuzzy skill multimap is not necessary a simple closure space or knowledge space. Now, we give characterizations of fuzzy skill multimaps such that the delineated knowledge structures are knowledge space, simple closure space, learning space, discriminative knowledge structure and bi-discriminative knowledge structure respectively.}
\end{rem}

We provide a basic concept of {\it molecule} as follows for characterizations of knowledge space and learning space.

\begin{defn}
{\rm Let~$(Q, S, \mu)$ be a fuzzy skill multimap such that~$\mu(q)$ is finite for each~$q\in Q$. The element $T\in \mathscr{F}(S)$ is called a {\it molecule} if there is~$s_{0}\in S$ such that~$T(s_{0})\neq 0$, and for each $s\in S\backslash \{s_{0}\}$,~$T(s)=0$. For convenience, we denote the molecule~$T$ by~$\frac{T(s_{0})}{s_{0}}$.}
\end{defn}

\begin{Notation}
Suppose~$(Q, S, \mu)$ is a fuzzy skill multimap. Throughout this paper, for any~$q\in Q$, we denote the molecule~$T^{C}_{q}$ by~$\frac{T^{C}_{q}(s)}{s}$ if there is~$C\in \mu(q)$ such that molecule~$T^{C}_{q}(s)=C(s)>0$. Moreover, for every $q\in Q$, we denote $\mu_{M}(q)$ by the set of minimum elements in $\mu(q)$.
\end{Notation}

\begin{rem}~\label{bbbb}
{\rm A fuzzy skill multimap~$(Q, S, \mu)$ is {\it disjunctive} if for each~$q\in Q$ the competencies~$C\in \mu(q)$ are molecules, and $(Q, S, \mu)$ is {\it conjunctive} if, for each~$q\in Q$, $\mu(q)=\{C\}$ for some~$C\in \mathscr{F}(S)$. Clearly, each conjunctive fuzzy skill multimap is a fuzzy skill function, but not vice versa.}
\end{rem}

By similar proofs of \cite[Theorem~36 and Corollay 7]{LCL}, following theorem and corollary hold.

\begin{thm}\label{ttttt2}
Assume $(Q, S, \mu)$ is a fuzzy skill multimap, where each $\mu(q)$ is a finite set. Then the delineate knowledge structure $(Q, \mathscr{K})$ is a knowledge space iff, for any $K\subseteq Q$, $K\in\mathscr{K}$ iff there exists $\mathscr{P}_{K}\subseteq\bigcup_{q\in Q}\mu_{M}(q)$ with $K=\bigcup_{D\in\mathscr{P}_{K}}p(D)$.
\end{thm}

\begin{cor}\label{2021}
Let $(Q, S, \mu)$ be a fuzzy skill multimap, where each $\mu(q)$ is a finite set. If, for any subset $Q^{\prime}\subseteq Q$ and $g\in Q$, the following condition ($\star$) holds, then $(Q, \mathscr{K})$ is a knowledge space.

\smallskip
($\star$) For each subfamily $\mathscr{P}$ of $\bigcup_{q\in Q^{\prime}}\mu_{M}(q)$, if $C\nsubseteq D$ for any $C\in \mu_{M}(g)$ and $D\in \mathscr{P}$, then  $C\nsubseteq \bigcup\mathscr{P}$ for any $C\in \mu_{M}(g)$.
\end{cor}

\begin{thm}\label{1-2}
Let~$(Q, S, \mu)$ be a fuzzy skill multimap. If for each~$q\in Q$ and~$C\in \mu(q)$, there exists a molecule~$\frac{T^{C}_{q}(s)}{s}\subseteq C$ such that~$\frac{T^{C}_{q}(s)}{s}\in \mu(q)$, then the delineated knowledge structure~$(Q,\mathscr{K})$ is a knowledge space.
\end{thm}

\begin{proof}
Let $(Q, \mathscr{K})$ be delineated by a fuzzy skill multimap~$(Q, S, \mu)$, and let~$p$ be the problem function induced by~$(Q, S, \mu)$. Take any family $\{T_{i}:i\in I\}\subset 2^{\mathscr{F}(S)\backslash \emptyset}$; then we conclude that~$\bigcup_{i\in I}p(T_{i})=p(\bigcup_{i\in I}T_{i})$. In fact, it clear that~$\bigcup_{i\in I}p(T_{i})\subseteq p(\bigcup_{i\in I}T_{i})$. It need to show that~$p(\bigcup_{i\in I}T_{i})\subseteq \bigcup_{i\in I}p(T_{i})$. Take any~$q\in p(\bigcup_{i\in I}T_{i})$. Then there exists~$C\in \mu(q)$ such~ that $C\subset \bigcup_{i\in I}T_{i}$. If there exists $T_{i}$ such that $C\subseteq T_{i}$, then it is obvious that $q\in p(T_{i})\subseteq \bigcup_{i\in I}p(T_{i})$. Consequently,~$\bigcup_{i\in I}p(T_{i})=p(\bigcup_{i\in I}T_{i})$. Suppose not, for any~$i\in I$ we have~$C\not\subseteq T_{i}$. However, there exists a molecule~$\frac{T^{C}_{q}(s_{C})}{s_{C}}$ such that~$\frac{T^{C}_{q}(s_{C})}{s_{C}}\subseteq C$ and~$\frac{T^{C}_{q}(s_{C})}{s_{C}}\in\mu(q)$. Since~$C\subset\bigcup_{i\in I}T_{i}$, there exists~$i\in I$ such that ~$\frac{T^{C}_{q}(s_{C})}{s_{C}}\subseteq T_{i}$, hence~$q\in p(T_{i})\subseteq \bigcup_{i\in I}P(T_{i})$.
\end{proof}

\begin{cor}
Assume~$(Q, S, \mu)$ is a skill multimap. If, for each~$q\in Q$ and~$C\in \mu(q)$, there exists~ $s_{C}\in C$ such that~$\{s_{C}\}\in \mu(q)$, then the delineated knowledge structure~$(Q, \mathscr{K})$ is a knowledge space.
\end{cor}

Theorems~\ref{ttttt2} and \ref{1-2} give partial answers to Question~\ref{1-5}. However, the reverse of Theorem~\ref{1-2} does not hold, see the following example.

\begin{ex}
{Let~$Q=\{q_{1}, q_{2}\}$,~$S$ be a non-empty skill set, and a fuzzy skill multimap be defined by
\begin{equation*}
\begin{aligned}
&\mu(q_{1})=\{S\}, &\mu(q_{2})=\{S\}.\\
\end{aligned}
\end{equation*}
Then the delineated knowledge structure
$$\mathscr{K}=\{\emptyset, Q\}$$
is a knowledge space. However, there is~$q_{1}\in Q$ such that, for any~$C\in \mu(q_{1})$, the competence~$C$ is not a molecule.}
\end{ex}

\begin{thm}~\label{aa}
Assume~$(Q, S, \mu)$ is a fuzzy skill multimap such that each $\mu(q)$ is finite. If for each~$q\in Q$ and any $C_{1}, C_{2}\in \mu(q)$, there exists~$C_{3}\in \mu(q)$ such that~$C_{3}\subseteq C_{1}\cap C_{2}$, then the delineated knowledge structure~$(Q, \mathscr{K})$ is a simple closure space.
\end{thm}

\begin{proof}
Take any subfamily~$\{T_{i}: i\in I\}\subseteq 2^{\mathscr{F}(S)\backslash \emptyset}$. We claim that~$\bigcap_{i\in I}p(T_{i})=p(\bigcap_{i\in I}T_{i})$. In fact, it clear that~$p(\bigcap_{i\in I}T_{i})\subseteq \bigcap_{i\in I}p(T_{i})$. Hence it suffices to prove that~$\bigcap_{i\in I}p(T_{i})\subseteq p(\bigcap_{i\in I}T_{i})$. Indeed, pick any~$q\in \bigcap_{i\in I}p(T_{i})$. Then~$q\in p(T_{i})$, hence there exists~$C_{i}\in \mu(q)$ such that~$C_{i}\subseteq T_{i}$. Since~$\mu(q)$ is finite, there exists~$C\in \mu(q)$ such that~$C\subseteq \bigcap_{i\in I}C_{i}$ by our assumption. Then~$C\subseteq \bigcap_{i\in I}C_{i}\subseteq \bigcap_{i\in I}T_{i}$. Hence $p\in p(\bigcap_{i\in I}T_{i})$, then $\bigcap_{i\in I}p(T_{i})\subseteq p(\bigcap_{i\in I}T_{i})$. Consequently, the delineated knowledge structure~$(Q, \mathscr{K})$ is a simple closure space.
\end{proof}

\begin{cor}
Assume $(Q, S, \mu)$ is a skill multimap such that each $\mu(q)$ is finite. If for each~$q\in Q$ and~$C_{1}, C_{2}\in \mu(q)$, 
there exists~$C_{3}\in \mu(q)$ with $C_{3}\subseteq C_{1}\cap C_{2}$, then the delineated knowledge structure~$(Q, \mathscr{K})$ 
is a simple closure space.
\end{cor}

Theorem~\ref{aa} gives a partial answer to Question~\ref{1-5}. Nevertheless, the reverse of Theorem~\ref{aa} is not true, see the following example.

\begin{ex}
{Let~$Q=\{q_{1}, q_{2}, q_{3}\}$ and~$S=\{s_{1}, s_{2}\}$. And a fuzzy skill multimap be defined by
\begin{equation*}
\begin{aligned}
&\mu(q_{1})=\{\{\frac{0.2}{s_{1}}\}, \{\frac{0.1}{s_{1}}, \frac{0.3}{s_{2}}\}\},
&\mu(q_{2})=\{\{\frac{0.6}{s_{1}}, \frac{0.7}{s_{2}}\}\}, \
&\mu(q_{3})=\{\{\frac{0.4}{s_{1}}\}\}.
\end{aligned}
\end{equation*}
Then the delineated knowledge structure
$$\mathscr{K}=\{\emptyset, \{q_{1}\}, \{q_{1}, q_{3}\}, Q\}$$
is a simple closure space. However, the intersection of the competence~$\{\frac{0.2}{s_{1}}\}$ and the competence~$ \{\frac{0.1}{s_{1}}, \frac{0.3}{s_{2}}\}$ is not in~$\mu(q_{1})$.}
\end{ex}

From Theorems~\ref{1-2} and \ref{aa}, the following corollary is immediate.

\begin{cor}
Assume $(Q, S, \mu)$ is a fuzzy skill multimap such that each $\mu(q)$ is finite. If for each~$q\in Q$, there exists a molecule~$\frac{T_{q}(s)}{s}\in \mu(q)$ such that~$\frac{T_{q}(s)}{s}\in \mu(q)$ is a minimum element in $\mu(q)$, then the delineated knowledge structure~$(Q, \mathscr{K})$ is a quasi ordinal space.
\end{cor}

\begin{cor}
Assume $(Q, S, \mu)$ is a skill multimap such that $\mu(q)$ is finite. If for each~$q\in Q$, there exists~$\{s\}\in \mu(q)$ such that~$\{s\}$ is a minimum element in $\mu(q)$, then the delineated knowledge structure~$(Q, \mathscr{K})$ is a quasi ordinal space.
\end{cor}

Assume $(Q,S,\mu)$ is a fuzzy skill multimap. We always denote the following set

$$\{T: T\ \mbox{is a molecule}\ \mbox{and}\ T\in\bigcup_{q\in Q}\mu(q)\}$$ by $\mathscr{D}$ in this paper.
For any~$T\in \mathscr{D}$, put
$$[T]=\{q\in Q: \mbox{there exists}\ C\in \mu(q)\ \mbox{with}\ C\subseteq T\}.$$
For any subfamily~$\mathscr{D}'\subseteq \mathscr{D}$, denote
$$[\mathscr{D}']=\underset{F\in \mathscr{D}'}{\bigcup}[F].$$

\begin{thm}\label{ff}
Assume $(Q, S, \mu)$ is a fuzzy skill multimap such that each $\mu(q)$ is finite. If the following conditions~(1)-(3) hold, then the delineated knowledge structure~$(Q,\mathscr{K})$ is a learning space.

\smallskip
(1)~$Q$ is finite;

\smallskip
(2) For each~$q\in Q$ and~$C\in \mu(q)$, there exists a molecule~$\frac{T^{C}_{q}(s^{C}_{q})}{s^{C}_{q}}\in \mu(q)$;

\smallskip
(3) For each~$T\in \mathscr{D}$, if~$|[T]|\geq 2$, then there is~$\mathscr{D}'\subseteq \mathscr{D}$, such that~$|[T]\backslash [\mathscr{D}']|=1$ and~$[\mathscr{D}']\subseteq [T]$.
\end{thm}

\begin{proof}
By theorem~\ref{1-2}, $(Q,\mathscr{K})$ is a knowledge space. Moreover, by condition~(2), for any $K\in\mathscr{K}$ there exists a minimum set~$S(K)\subseteq S$ such that~$K=p(T)$, where~$T=\{\frac{T_{s}(s)}{s}: s\in S(K)\}$. By \cite[Theorem 2.2.4]{falmagne2011learning}, it suffices to prove that~$(Q,\mathscr{K})$ is well-graded. In fact, it only need to prove the following claim.

\smallskip
{\bf Claim:}  For each state~$K\in \mathscr{K}$, there exists~$q\in Q$ with~$K\backslash \{q\}\in \mathscr{K}$.

\smallskip
We shall prove claim by induction on the cardinality of any states of $\mathscr{K}$. If~$|K|\leq 2$, the conclusion is obvious by the assumption of~(3). Assume that for each $K\in\mathscr{K}$ with~$|K|\leq n$ there is~$q\in Q$ such that~$K\backslash \{q\}\in \mathscr{K}$. Now take any $K\in\mathscr{K}$ with~$|K|=n+1$. Then there exists a minimum set~$S(K)\subseteq S$ such that~$K=p(T)$, where~$T=\{\frac{T_{s}(s)}{s}: s\in S(K)\}$.

If there exists~$s_{K}\in S(K)$ such that~$|[\frac{T_{s_{K}}(s(K))}{s_{K}}]|=1$, then $$K\backslash [\frac{T_{s_{K}}(s_{K})}{s_{K}}]=p(\{\frac{T_{s}(s)}{s}: s\in S(K)\backslash \{{s_{K}}\}\})\in \mathscr{K}$$ by the minimality of~$S(K)$. Then the proof is complete. Now we can assume that $|[\frac{T_{s}(s)}{s}]|\geq 2$ for each~$s\in S(K)$.
If~$|S(K)|=1$, then conclusion is obvious by the assumption of (3). By the minimality of~$S(K)$, if ~$|S(K)|\geq 2$, then~$p(\{\frac{T_{s}(s)}{s}:s\in S(K)\backslash \{s_{K}\}\})\in \mathscr{K}$ for each~$s_{K}\in S(K)$ and~$|p(\{\frac{T_{s}(s)}{s}: s\in S(K)\backslash \{s_{K}\}\})|\leq n$. Fix $s_{K}\in S(K)$. By our assumption, there exists~$q(s_{K})\in p(\{\frac{T_{s}(s)}{s}: s\in S(K)\backslash \{s_{K}\}\})$ and subfamily~$\mathscr{B}\subseteq \mathscr{D}$ satisfying~$p(\{\frac{T_{s}(s)}{s}:s\in S(K)\backslash \{s_{K}\}\})\backslash \{q(s_{K})\}=p({\mathscr{B}})$. If~$q(s_{K})\in [\frac{T_{s_{K}}(s_{K})}{s_{K}}]$ and there exists~$\mathscr{A}\subseteq \mathscr{D}$ such that~$[\frac{T_{s_{K}}(s_{K})}{s_{K}}]\backslash \{q(s_{K})\}=p(\mathscr{A})$, then it follows that
\begin{equation*}
\displaystyle
p(T)\backslash \{q(s_{K})\}=p(\mathscr{A}\cup \mathscr{B})=p(\mathscr{A})\cup p(\mathscr{B}),
\end{equation*}
hence the proof is completed. Suppose not, we assume that $q(s_{K})\in [\frac{T_{s_{K}}(s_{K})}{s_{K}}]$ such that  $[\frac{T_{s_{K}}(s_{K})}{s_{K}}]\setminus\{q(s_{K})\}\neq P(\mathscr{A})$ for any $\mathscr{A}\subseteq\mathscr{D}$.
Since~$Q$ is finite, it follows that~$[\frac{T_{s_{K}}(s_{K})}{s_{K}}]$ and~$p(\mathscr{B})$ are also finite. So we can do the same way on~$\mathscr{B}$ as above. From the minimality of $S(K)$, after finitely many steps we get~$q\in p(\{\frac{T_{s}(s)}{s}:s\in S(K)\backslash \{s_{K}\}\})$ and~$\mathscr{C}_{1}, \mathscr{C}_{2}\subseteq \mathscr{D}$ satisfying the conditions~$(a)$-$(c)$ as follows:

\smallskip
(a)~$p(\mathscr{C}_{1})\backslash \{q\}=p(\mathscr{C}_{2})$;

\smallskip
(b) either~$q\not\in [\frac{T_{s_{K}}(s_{K})}{s_{K}}]$ or~$q\in [\frac{T_{s_{K}}(s_{K})}{s_{K}}]$ and there exists ~$\mathscr{A}\subset\mathscr{D}$ such that~$[\frac{T_{s_{K}}(s_{K})}{s_{K}}]\backslash \{q\}=P(\mathscr{A})$:

\smallskip
(c)~$P(\mathscr{C}_{1})\cup[\frac{T_{s_{K}}(s_{K})}{s_{K}}]=P(T)$.
\smallskip

If~$q\notin [\frac{T_{s_{K}}(s_{K})}{s_{K}}]$, then it follows from (a) and (c) that~$p(\mathscr{C}_{2})\cup [\frac{T_{s_{K}}(s_{K})}{s_{K}}]=p(T)\backslash \{q\}\in \mathscr{K}$; If~$q\in [\frac{T_{s_{K}}(s_{K})}{s_{K}}]$ and there exists~$\mathscr{A}\subseteq \mathscr{D}$ such that~$[\frac{T_{s_{K}}(s_{K})}{s_{K}}]\backslash \{q\}=p(\mathscr{A})$, then it follows from (a) and (c) that~$P(\mathscr{C}_{2})\cup p(\mathscr{A})=p(T)\backslash \{q\}\in \mathscr{K}$. So we can derive that~$(Q, \mathscr{K})$ is well-graded.

Consequently, Claim is proved and thus~$(Q, \mathscr{K})$ is a learning space.
\end{proof}

\begin{cor}\label{cc}
Assume $(Q, S, \mu)$ is a skill multimap such that each $\mu(q)$ is finite. If the following conditions~(1)-(3) hold, then the delineated knowledge structure~$(Q, \mathscr{K})$ is a learning space.

\smallskip
$(1)$~$Q$ is finite;

\smallskip
$(2)$ For each~$q\in Q$ and~$C\in \mu(q)$, there exists~$s_{C}\in C$ such that ${s_{C}}\in \mu(q)$;

\smallskip
$(3)$ For each~$s\in S$, if~$|[s]|\geq 2$, then there is~$S'\subseteq S$ such that~$|[s]\backslash [S']|=1$ and~$[S']\subseteq [s]$.
\end{cor}

Theorem~\ref{ff} gives a partial answer to Question~\ref{1-5}. Moreover,
the delineated knowledge structure~$(Q,\mathscr{K})$ is an antimatroid if~$(Q, \mathscr{K})$ is a learning space from~\cite[Theorem 2.2.4]{falmagne2011learning}. However, the conditions (1)-(3) in Theorem~\ref{ff} and Corollary~\ref{cc} respectively are sufficient and not necessary, see the following examples.

\begin{ex}
{Let~$Q=\{q_{1}, q_{2}\}$ and~$S=\{s_{1}, s_{2}, s_{3}\}$. A fuzzy skill multimap be defined by
\begin{equation*}
\begin{aligned}
&\mu(q_{1})=\{\{\frac{0.1}{s_{1}}, \frac{0.3}{s_{2}}\}, \{\frac{0.4}{s_{2}}, \frac{0.6}{s_{3}}\}\},
&\mu(q_{2})=\{\{\frac{0.4}{s_{2}}, \frac{0.6}{s_{3}}\}\}.
\end{aligned}
\end{equation*}
Then the delineated knowledge structure
$$\mathscr{K}=\{\emptyset, \{q_{1}\}, Q\}$$
is a learning space. However, any molecule~$\frac{T^{C}_{q}(s^{C}_{q})}{s^{C}_{q}}\subseteq C$ for each~$C\in \mu(q_{1})$ is not contained in~$\mu(q_{1})$. The other case is analogous. Clearly, the reverse of Theorem~\ref{ff} is not true. Indeed, the condition (2) in the assumption of Theorem~\ref{ff} does not hold.}
\end{ex}

\begin{ex}
{Let~$Q=\{q_{1}, q_{2}\}$ and~$S=\{s_{1}, s_{2}, s_{3}\}$. A fuzzy skill multimap be defined by
\begin{equation*}
\begin{aligned}
&\mu(q_{1})=\{\{s_{1}, s_{2}\}, \{s_{2}, s_{3}\}\},
&\mu(q_{2})=\{\{s_{2}, s_{3}\}\}.
\end{aligned}
\end{equation*}
Then the delineated knowledge structure
$$\mathscr{K}=\{\emptyset, \{q_{1}\}, Q\}$$
is a learning space. For any item~$s_{C}\in C$ for each~$C\in \mu(q_{1})$, the set~$\{s_{C}\}$ is not contained in~$\mu(q_{1})$ though. The other case is analogous. Clearly, the reverse of Corollary~\ref{cc} is not true. Indeed, the condition (2) in the assumption of Corollary~\ref{cc} does not hold.}
\end{ex}

\begin{defn}[\cite{falmagne2011learning}]
{\rm Suppose that $(Q, \mathscr{K})$ is a discriminative knowledge structure. For every $K\in\mathscr{K}$, the set $K^{\mathscr{I}}=\{t\in K: K\setminus\{t\}\in\mathscr{K}\}$ is called the {\it inner fringe} of $K$, and the set $K^{\mathscr{O}}=\{t\in Q\setminus K: K\cup\{t\}\in\mathscr{K}\}$ is called the {\it outer fringe} of $K$. Further, the union of the inner and outer fringes of $K$ is called the {\it fringe} of $K$, which is denoted by $$K^{\mathscr{F}}=K^{\mathscr{I}}\cup K^{\mathscr{O}}.$$}
\end{defn}

Assume $(Q, \mathscr{K})$ is a knowledge structure delineated by a fuzzy skill multimaps~$(Q, S, \mu)$ which satisfies the assumption in Theorem~\ref{1-2}. Then, for each state~$K\in \mathscr{K}$, there exists a minimal fuzzy skill set~$T_{min}\in \mathscr{F}(S)$ such that~$K=p(T_{min})$ from Theorem~\ref{1-2}.

\begin{prop}
Assume $(Q, \mathscr{K})$ is a knowledge structure delineated by a fuzzy skill multimap~$(Q, S, \mu)$ which satisfies the assumption in Theorem~\ref{1-2}. Then, for each~$K\in \mathscr{K}$ and~$q\in Q\backslash K$, we have~$q\in K^{\mathscr{O}}$ iff there exists a molecule~$\frac{T_{q}(s_{q})}{s_{q}}\in \mu(q)$ such that~$[\frac{T_{q}(s_{q})}{s_{q}}]\backslash K=\{q\}$.
\end{prop}

\begin{proof}
Sufficiency. Take any~$K\in \mathscr{K}$. From Theorem~\ref{1-2}, there exists a minimal fuzzy skill set~$T_{min}\in \mathscr{F}(S)$ such that~$K=p(T_{min})$. From the assumption, there exists a molecule~$\frac{T_{q}(s_{q})}{s_{q}}\in \mu(q)$ such that~$[\frac{T_{q}(s_{q})}{s_{q}}]\backslash K=\{q\}$, then
\begin{equation*}
\displaystyle
p(T_{min})\bigcup\{q\}=p(T_{min})\bigcup p(\frac{T_{q}(s_{q})}{s_{q}})=p(T_{min}\cup \frac{T_{q}(s_{q})}{s_{q}}).
\end{equation*}
It follows that~$K\bigcup\{q\}=p(T_{S(K)}\bigcup \frac{T_{q}(s_{q})}{s_{q}})\in \mathscr{K}$.

Necessity. Assume that~$q\in K^{\mathscr{O}}$, then~$K\bigcup \{q\}\in \mathscr{K}$ and $q\not\in K$. Suppose that~$|[\frac{T(s)}{s}]\backslash K|\neq 1$ for each molecule~$\frac{T(s)}{s}\in \mu(q)$. Then, since $q\not\in K$, it follows that $|[\frac{T(s)}{s}]\backslash K|\geq 2$. However, since~$(Q, \mathscr{K})$ is a knowledge space by Theorem~\ref{1-2},  then there exists a molecule~$\frac{T_{s_{0}}(s_{0})}{s_{0}}\in \mu(q)$ such that $K\bigcup \{q\}= K\cup p(\frac{T_{s_{0}}(s_{0})}{s_{0}})\in \mathscr{K}$, hence  $|[\frac{T_{s_{0}}(s_{0})}{s_{0}}]\backslash K|=1$, which is a contradiction.
\end{proof}

\begin{prop}
Assume~$(Q, \mathscr{K})$ is a knowledge space delineated by a fuzzy skill multimap~$(Q, S, \mu)$ which satisfies the assumption in Theorem~\ref{1-2}. Then, for every~$K\in \mathscr{K}$ and~$q\in Q\backslash K$, we have~$r\in K^{\mathscr{I}}$ iff for each~$q\in K\backslash \{r\}$, there exists a molecule~$\frac{T_{q}(s_{q})}{s_{q}}\in \mu(q)$ such that~$[\frac{T_{q}(s_{q})}{s_{q}}]\subseteq K\backslash \{r\}$.
\end{prop}

\begin{proof}
Sufficiency. Take any~$K\in \mathscr{K}$. Assume that for any~$q\in K\backslash \{r\}$, there exists a molecule~$\frac{T_{q}(s_{q})}{s_{q}}\in \mu(q)$ such that~$[\frac{T_{q}(s_{q})}{s_{q}}]\subseteq K\backslash \{r\}$. From Theorem~\ref{1-2},~$K\backslash \{r\}=\bigcup_{q\in K\backslash \{r\}}p(\frac{T_{q}(s_{q})}{s_{q}})=p(\bigcup_{q\in K\backslash \{r\}}\frac{T_{q}(s_{q})}{s_{q}})$. It clear that~$K\backslash \{r\}\in \mathscr{K}$ and~$r\in K^{\mathscr{I}}$.

Necessity. Assume that~$r\in K^{\mathscr{I}}$, then~$K\backslash \{r\}\in \mathscr{K}$ and $r\in K$. Suppose there exists~$q_{0}\in K\backslash \{r\}$ such that for any molecule~$\frac{T_{q_{0}}(s_{q_{0}})}{s_{q_{0}}}\in \mu(q_{0})$,~$[\frac{T_{q_{0}}(s_{q_{0}})}{s_{q_{0}}}]\backslash (K\backslash \{r\})\neq \emptyset$. However, since~$(Q, \mathscr{K})$ is a knowledge space by Theorem~\ref{1-2}, then for each~$q\in K\backslash \{r\}$, there exists molecule~$\frac{T_{q}(s_{q})}{s_{q}}\in \mu(q)$ such that~$K\backslash \{r\}=\bigcup_{q\in K\backslash \{r\}} p(\frac{T_{q}(s_{q})}{s_{q}})=p(\bigcup_{q\in K\backslash \{r\}}(\frac{T_{q}(s_{q})}{s_{q}}))$. This a contradiction with our assumption.
\end{proof}

\smallskip
\subsection{The separability of items in knowledge structures delineated by fuzzy skill multimaps}
In order to give the necessary and sufficient conditions such that the delineated knowledge structures by fuzzy skill multimaps are discriminative and bi-discriminative respectively, this section gives an introduction to the concept of refining a family of fuzzy subsets by a similar definition in Ge and Li~\cite{ref4}.

\begin{defn}
{\rm Assume $\mathscr{U}$ and~$\mathscr{V}$ are two family of fuzzy sets of a non-empty set $X$. If for each~$U\in \mathscr{U}$ there exists $V\in \mathscr{V}$ such that $V\subseteq U$, then we say that~$\mathscr{U}$ is {\it refined by~$\mathscr{V}$}, and is denoted by~$\mathscr{U}\preccurlyeq \mathscr{V}$; if not, $\mathscr{U}$ is not refined by $\mathscr{V}$, and is denoted by~$\mathscr{U}\not\preccurlyeq \mathscr{V}$. Especially, if~$\mathscr{U}=\{U\}$, then~$U$ is (resp., is not) {\it refined by~$\mathscr{V}$}, and is denoted by~ $U\preccurlyeq \mathscr{V}$ (resp.,~$U\not\preccurlyeq \mathscr{V}$).}
\end{defn}

\begin{rem}
For a fuzzy skill multimap~$(Q,S,\mu)$ and any $q, r \in Q$, it is easy to see that $\mu(q)\not\preccurlyeq \mu(r)$ implies that there is~$C_{q}\in \mu(q)$ such that~$C_{q}\not\preccurlyeq \mu(r)$, i.e., for any~$C_{r}\in \mu(r)$, there exists~$s_{C_{r}}\in S$ such that~$C_{q}(s_{C_{r}})<C_{r}(s_{C_{r}})$.
\end{rem}

The following theorem, which gives a partial answer to Question~\ref{kkk}, provides a necessary and sufficient condition such that the knowledge structures delineated by fuzzy skill multimaps are discriminative.

\begin{thm}~\label{fff}
Assume $(Q,S,\mu)$ is a fuzzy skill multimap such that $(Q, \mathscr{K})$ is the knowledge structure delineated by~$(Q,S,\mu)$. Consider the following statements.

\smallskip
(1)~$(Q,\mathscr{K})$ is discriminative.

\smallskip
(2)~for any~$q,r\in Q$ with~$q\neq r$, we have $\mu(r)\not\subseteq \mu(q)$ or~$\mu(q)\not\subseteq \mu(r)$.

\smallskip
(3) for any~$q,r\in Q$ with~$q\neq r$, we have $\mu(q)\neq \mu(r)$.

\smallskip
(4)~for any~$q,r\in Q$ with~$q\neq r$, we have $\mu(r)\not\preccurlyeq \mu(q)$ or~$\mu(q)\not\preccurlyeq \mu(r)$.

Then~$(1)\Longleftrightarrow (4)\Longrightarrow (2)\Longleftrightarrow (3)$. Further, if~$(Q, S, \mu)$ is a fuzzy skill function, then~$(1)\Longleftrightarrow (2)\Longleftrightarrow (3)\Longleftrightarrow (4)$
\end{thm}

\begin{proof}
Obviously, $(2)\Longleftrightarrow (3)$. Now we need to prove~$(1)\Longleftrightarrow (4)$, and prove $(4)\Longrightarrow (2)$ and~$(2)\Longrightarrow (1)$ when $(Q,S,\mu)$ is a fuzzy skill function. Assume $p$ is the problem function induced by fuzzy skill multimap~$(Q, S, \mu)$.

\smallskip
~$(1)\Longrightarrow (4).$ Let $(Q,\mathscr{K})$ be discriminative. Then $(Q,\mathscr{K})$ is a ~$T_{0}$-knowledge structure by (1) of Lemma~\ref{1-1}. Take any $q, r \in Q$ with $q\neq r$. Then there exists~$K\in \mathscr{K}$ such that $K\cap\{q, r\}$ is just an one-point set. Without loss of generality, we may assume that ~$K\cap\{q, r\}=\{q\}$. Because $(Q,\mathscr{K})$ is delineated by $(Q,S,\mu)$, there exists~$T\in \mathscr{F}(S)$ with~$K=p(T)$, which shows that there is~$C_{q}\in \mu(q)$ with $C_{q}\subseteq T$. Now we only claim that~ $C_{q}\not\preccurlyeq \mu(r)$. Indeed, suppose not, then $C_{q}\preccurlyeq \mu(r)$, hence we can take~$C_{r}\in \mu(r)$ with $C_{r}\subseteq C_{q}\subseteq T$. This implies that~$r\in p(T)=K$. However,  $r\notin K$, this is a contradiction.

\smallskip
~(4)$\Longrightarrow (1).$ Let~$(4)$ be true. Then we need to prove that~$(Q,\mathscr{K})$ is a~ $T_{0}$-knowledge structure by (1) of Lemma~\ref{1-1}. Take any ~$q, r \in Q$ with~$q\neq r$. Without loss of generality, we may assume~$\mu(q)\not\preccurlyeq \mu(r)$, which implies that there exists~$C_{q}\in \mu(q)$ such that~$C_{q}\not\preccurlyeq \mu(r)$, i.e., for any~$C_{r}\in \mu(r)$, there exists~$s_{C_{r}}\in S$ such that~$C_{q}(s_{C_{r}})<C_{r}(s_{C_{r}})$. Obviously, we have $q\in p(T)$, where~$T=C_{q}$. Now it needs to show~$r\notin p(T)$. Indeed, if~$r\in p(T)$, then there exists~$C_{r}\in \mu(r)$ such that~$C_{r}(s)\leq T(s)=C_{q}(s)$ for all~$s\in S$. However, $C_{q}\not\preccurlyeq \mu(r)$, this is a contradiction. Thus~$r\notin p(T)$.

\smallskip
~(4)$\Longrightarrow (2).$ Without loss of generality, suppose that~$\mu(q)\not\preccurlyeq \mu(r)$. Then there exists~$C_{q}\in \mu(q)$ such that~$C_{q}\not\preccurlyeq \mu(r)$, i.e., for any~$C_{r}\in \mu(r)$ there exists~$s_{C_{r}}\in S$ such that~$C_{q}(s_{C_{r}})<C_{r}(s_{C_{r}})$. Clearly,~$C_{q}\notin \mu(r)$. Hence~$\mu(q)\not\subseteq \mu(r)$.

\smallskip
~(2)$\Longrightarrow (1).$ Assume~$(Q,S,\mu)$ is a fuzzy skill function, and assume (2) holds. We need to prove that~$(Q,\mathscr{K})$ is a~$T_{0}$-knowledge structure by (1) of Lemma~\ref{1-1}. Take any~$q, r \in Q$ with $q\neq r$. Without loss of generality, we assume~$\mu(q)\not\subseteq \mu(r)$, then $\mu(q)\backslash \mu(r)\neq \emptyset$. Choose any $C_{q}\in \mu(q)\backslash \mu(r)$. Then there is~$S_{0}\subseteq S$ such that~$C_{q}(s)>0$ for each~$s\in S_{0}$ and~$C_{q}(s)=0$ for each~$s\in S\backslash S_{0}$. Hence~$q\in p(C_{q})\in \mathscr{K}$. Clearly, the proof is completed if~$r\notin p(C_{q})$. Then we assume~$r\in p(C_{q})$, hence there exists~$C_{r}\in \mu(r)$ with~$C_{r}\subseteq C_{q}$. From $C_{q}\notin \mu(r)$, it follows that~$C_{r}(s)<C_{q}(s)$ for some~$s\in S_{0}$. Clearly,~$r\in p(C_{r})\in \mathscr{K}$. Therefore, now it only need to prove that~$q\notin p(C_{r})$. Indeed, if~$q\in p(C_{r})$, then there is $C'_{q}\in \mu(q)$ with $C'_{q}\subseteq C_{r}$. Hence $C'_{q}$,  $C_{q}\in \mu(q)$,~$C'_{q}\subseteq C_{q}$ and~$C'_{q}(s)<C_{q}(s)$ for some ~$s\in S_{0}$. However, elements of $\mu(q)$ are pairwise incomparable, this is a contradiction. Therefore, $q\notin p(C_{r})$.
\end{proof}

\vspace{3mm}
By Remark~\ref{bbbb}and Theorem~\ref{fff}, the following two corollaries are obvious.

\begin{cor}
A conjunctive fuzzy skill multimap is injective iff the delineated knowledge structure is discriminative.
\end{cor}

\begin{cor}~\label{gggg}
A disjunctive fuzzy skill multimap is injective whenever the delineated knowledge structure is discriminative.
\end{cor}

We give the following examples to show that the assumption~``$(Q,S,\mu)$ is a fuzzy skill function'' of $(3)\Longrightarrow (1)$ in Theorem~\ref{fff} can not be omitted, and the reverse of Corollary~\ref{gggg} does not hold.

\begin{ex}
Let~$Q=\{a,b\}$ and~$S=\{s_{1},s_{2}\}$. A fuzzy skill multimap be defined by~$$\mu(a)=\{\{\frac{0.3}{s_{1}}\}, \{\frac{0.4}{s_{2}}\}, \{\frac{0.7}{s_{2}}, \frac{0.4}{s_{3}}\}\}, \mu(b)=\{\{\frac{0.3}{s_{1}}\},\{\frac{0.4}{s_{2}}\}\}.$$ Clearly, the delineated knowledge structure$$\mathscr{K}=\{\emptyset, \{a,b\}\}$$ is not discriminative. However, $\mu(a)\preccurlyeq \mu(b)$ and~$\mu(b)\preccurlyeq \mu(a)$.
\end{ex}

\begin{ex}
Let~$Q=\{a,b\}$ and~$S=\{s_{1},s_{2}\}$. A disjunctive fuzzy skill multimap be defined by~$$\mu(a)=\{\{\frac{0.6}{s_{1}}\}, \{\frac{0.8}{s_{1}}\}, \{\frac{0.7}{s_{2}}\}\}, \mu(b)=\{\{\frac{0.3}{s_{1}}\},\{\frac{0.4}{s_{2}}\}\}.$$ Clearly, the delineated knowledge structure$$\mathscr{K}=\{\emptyset, \{a,b\}\}$$ is not discriminative. However, $\mu$ is injective.
\end{ex}

For a conjunctive fuzzy skill function $(Q,S,\mu)$, we have $\mu(q)=\{C_{q}\}$ for every~$q\in Q$, where $C_{q}\in \mathscr{F}(S)\backslash \emptyset$; then, for each $q\in Q$, $C_{q}$ is the unique minimum of~$\mu(q)$. For any disjunctive fuzzy skill function~$(Q, S, \mu)$, the set $\mu(q)$ does not have a unique minimum element when $\mu(q)$ contains at least two molecules.

\vspace{3mm}
For each~$q\in Q$, considering minimum elements of $\mu(q)$. The following theorem gives a necessary and sufficient conditions such that the knowledge structures delineated by fuzzy skill multimaps are discriminative, which gives an answer to Question~\ref{kkk}.

\begin{thm}\label{ttt}
Assume~$(Q,\mathscr{K})$ is the knowledge structure delineated by fuzzy skill multimap $(Q,S,\mu)$. For each~$q\in Q$ and any $C\in\mu(q)$, assume that there exists a minimum element~$M_{q, C}\in \mu(q)$. Then (1) $\Leftrightarrow$ (2) $\Leftrightarrow$ (3) for the following three statements.

\smallskip
(1)~$(Q,\mathscr{K})$ is discriminative.

\smallskip
(2) for any pair items $q, r \in Q$ with $q\neq r$, either there exists $C\in\mu(q)$ with $M_{q, C}\neq M_{r, D}$ for any $D\in\mu(r)$, or there exists $D\in\mu(r)$ with $M_{r, D}\neq M_{q, C}$ for any $C\in\mu(q)$.

\smallskip
(3) for any pair items $q, r \in Q$ with $q\neq r$, either there exists $C\in\mu(q)$ with $M_{q, C}\not\preccurlyeq \{M_{r, D}: D\in\mu(r)\}$, or there exists $D\in\mu(r)$ with $M_{r, D}\not\preccurlyeq \{M_{q, C}: C\in\mu(q)\}$.
\end{thm}

\begin{proof}
Clearly, (3) $\Rightarrow$ (2). We need to prove (1)$\Longrightarrow$ (3) and (2)$\Longrightarrow$ (1).

\smallskip
(1)$\Longrightarrow$ (3). Let $(Q,\mathscr{K})$ be discriminative. Then $(Q, \mathscr{K})$ is a $T_{0}$-knowledge structure By (1) of Lemma~\ref{1-1}. Take  any~$q, r \in Q$ with~$q\neq r$. Because $(Q, \mathscr{K})$ is a $T_{0}$-knowledge structure, we can take~$K\in \mathscr{K}$ with $K\cap\{q, r\}=\{q\}$. Hence there is $T\in \mathscr{F}(S)$ with $K=p(T)$, which shows that there exists~$C_{q}\in \mu(q)$ such that~$C_{q}\subseteq T$. Furthermore,~$M_{q, C}\in \mu(q)$ and~$M_{q, C}\subseteq C_{q}\subseteq T$. Obviously,~$M_{q, C}\not\preccurlyeq \{M_{r, D}: D\in\mu(r)\}$. Indeed, if~$M_{q, C}\preccurlyeq \{M_{r, D}: D\in\mu(r)\}$, then there exists $D\in\mu(r)$ with $M_{r, D}\subseteq M_{q, C}$. Hence $r\in P(T)=K$; however, $r\notin K$, this is a contradiction.

\smallskip
(2)$\Longrightarrow$ (1). Suppose~$(2)$ is true. Now we need to prove that~$(Q,\mathscr{K})$ is a~ $T_{0}$-knowledge structure by (1) of Lemma~\ref{1-1}. Take any~$q, r \in Q$ with $q\neq r$. Without loss of generality, we may assume that there exists $C\in\mu(q)$ with $M_{q, C}\neq M_{r, D}$ for each $D\in\mu(r)$. Since~$M_{q, C}\in \mu(q)$, it follows that $q\in p(M_{q, C})$; moreover, there exists a subset~$S_{0}\subseteq S$ such that~$M_{q, C}(s)>0$ for each~$s\in S_{0}$ and~$M_{q, C}(s)=0$ for each~$s\in S\backslash S_{0}$. We claim that~$r\notin p(M_{q, C})$. Suppose not, assume that~$r\in p(M_{q, C})$, then there is~$D\in \mu(r)$ such that~$D\subseteq M_{q, C}$. Hence, we have~$M_{r, D}\subseteq D\subseteq M_{q, C}$ and~$M_{r, D}(s_{0})<M_{q, C}(s_{0})$ for some~$s_{0}\in S_{0}$ because~$M_{q, C}\neq M_{r, D}$. Clearly,~$r\in p(M_{r, D})\in \mathscr{K}$. It suffices to prove that~$q\notin p(M_{r, D})$. Indeed, if~$q\in p(M_{r, D})$, then there is~$C_{q}\in \mu(q)$ with~$C_{q}\subseteq M_{r, D}$. Since $M_{r, D}\subseteq D\subseteq M_{q, C}$ and $M_{r, D}(s_{0})<M_{q, C}(s_{0})$, it follows that $C_{q}$ is strictly smaller than $M_{q, C}$, which is a contradiction with the minimality of $M_{q, C}$. Therefore,~$r\notin p(M_{q, C})$.
\end{proof}

\begin{ex}
Assume $Q=\{a,b,c\}$ and~$S=\{s_{1},s_{2},s_{3}\}$. A fuzzy skill multimap be defined by
\begin{equation*}
\begin{aligned}
&\mu(a)=\{\{\frac{0.2}{s_{1}}\}, \{\frac{0.1}{s_{2}}\}, \{\frac{0.3}{s_{1}}, \frac{0.4}{s_{2}}\}\},\\
&\mu(b)=\{\{\frac{0.2}{s_{1}}\}, \{\frac{0.1}{s_{2}}\} \{\frac{0.5}{s_{1}}, \frac{0.2}{s_{3}}\},\{\frac{0.3}{s_{2}}, \frac{0.3}{s_{3}}\}\},\\
&\mu(c)=\{\{\frac{0.2}{s_{1}}\}, \{\frac{0.1}{s_{2}}\}, \{\frac{0.3}{s_{1}}, \frac{0.5}{s_{2}}, \frac{0.6}{s_{3}}\}\}.
\end{aligned}
\end{equation*}
It is easily checked that the delineated knowledge structure$$\mathscr{K}=\{\emptyset, \{a,b,c\}\}$$ is not discriminative. Therefore, it follows that for any~$C\in\mu(a)$, there exists~$D\in\mu(b)$ such that~$M_{a,C}=M_{b,D}$ and for any~$D\in\mu(b)$, there exists $C\in\mu(a)$ such that~$M_{b,D}=M_{a,C}$. Clearly, for any~$C\in\mu(a)$,~$M_{q, C}\preccurlyeq \{M_{r, D}: D\in\mu(b)\}$, and for any~$D\in\mu(b)$,~$M_{q, D}\preccurlyeq \{M_{r, C}: C\in\mu(a)\}$. The other case is analogous.
\end{ex}

 We have the similar results of Theorem~\ref{ttt} for bi-separability as follows. Theorems~\ref{tttt} and \ref{1-3} give partial answers to Question~\ref{kkk}.

\begin{thm}\label{tttt}
Assume $(Q,\mathscr{K})$ is the knowledge structure delineated by~fuzzy skill multimap $(Q,S,\mu)$. For each~$q\in Q$ and any $C\in\mu(q)$, assume that there exists a minimum element~$M_{q, C}\in \mu(q)$. Then (1) $\Leftrightarrow$ (2) $\Leftrightarrow$ (3) for the following three statements.

\smallskip
(1)~$(Q,\mathscr{K})$ is bi-discriminative.

\smallskip
(2) for any pair items $q, r \in Q$ with $q\neq r$, there exist $C\in\mu(q)$ and $D\in\mu(r)$ such that $M_{q, C}\neq M_{r, E}$ for any $E\in\mu(r)$ and $M_{r, D}\neq M_{q, F}$ for any $F\in\mu(q)$.

\smallskip
(3) for any pair items $q, r \in Q$ with $q\neq r$, there exist $C\in\mu(q)$ and $D\in\mu(r)$ such that $M_{q, C}\not\preccurlyeq \{M_{r, E}: E\in\mu(r)\}$ and $M_{r, D}\not\preccurlyeq \{M_{q, F}: F\in\mu(q)\}$.
\end{thm}

\begin{ex}
Assume $Q=\{a,b,c\}$ and~$S=\{s_{1},s_{2},s_{3}\}$. A fuzzy skill multimap be defined by
\begin{equation*}
\begin{aligned}
&\mu(a)=\{\{\frac{0.2}{s_{1}}\}, \{\frac{0.1}{s_{2}}\}, \{\frac{0.3}{s_{1}}, \frac{0.4}{s_{2}}\}\},\\
&\mu(b)=\{\{\frac{0.2}{s_{1}}\}, \{\frac{0.1}{s_{2}}\}, \{\frac{0.6}{s_{3}}\}, \{\frac{0.5}{s_{1}}, \frac{0.2}{s_{2}}\}.
\end{aligned}
\end{equation*}
It is easily checked that the delineated knowledge structure$$\mathscr{K}=\{\emptyset, \{b\}, \{a,b\}\}$$ is not bi-discriminative. Therefore, for any~$C\in\mu(a)$, there exists $D\in\mu(b)$ such that~$M_{a,C}=M_{b,D}$. It follows~$M_{a, C}\preccurlyeq \{M_{b, D}: D\in\mu(b)\}$. Clearly, for any~$C\in\mu(a)$,~$M_{q, C}\preccurlyeq \{M_{r, D}: D\in\mu(b)\}$.
\end{ex}

From Theorems~\ref{ttt} and~\ref{tttt}, the following corollary holds.

\begin{cor}\label{ccc}
Assume $(Q,\mathscr{K})$ is the knowledge structure delineated by~fuzzy skill multimap $(Q,S,\mu)$. Assume each $\mu(q)$ has a minimum $M_{q}$, then the following four statements are equivalent:

\smallskip
(1) $(Q,\mathscr{K})$ is discriminative;

\smallskip
(2) $(Q,\mathscr{K})$ is bi-discriminative;

\smallskip
(3) $M_{q}\neq M_{r}$ for any distinct items $q, r\in Q$;

\smallskip
(4) $\{M_{q}|q\in Q\}$ is pairwise incomparable.
\end{cor}

\begin{thm}\label{1-3}
Assume $(Q,\mathscr{K})$ is the knowledge structure delineated by~fuzzy skill multimap $(Q,S,\mu)$. Then the following statements are equivalent.

\smallskip
(1)~$(Q,\mathscr{K})$ is bi-discriminative.

\smallskip
(2) for any~$q,r\in Q$ with~$q\neq r$, we have $\mu(q)\not\preccurlyeq \mu(r)$ and~$\mu(r)\not\preccurlyeq \mu(q)$.
\end{thm}

\begin{proof}
(1) $\Longrightarrow$ (2). Let $(Q,\mathscr{K})$ be bi-discriminative. It follows from (2) of Lemma~\ref{1-1} that~$(Q,\mathscr{K})$ is a~$T_{1}$-knowledge structure. Take any~$q,r \in Q$ with~$q\neq r$. It suffices to prove that~$\mu(q)\not\preccurlyeq \mu(r)$ since the proof of~$\mu(r)\not\preccurlyeq \mu(q)$ is similar. Because~$(Q,\mathscr{K})$ is a~$T_{1}$-knowledge structure, there is~$K\in \mathscr{K}$ with $K\cap\{q, r\}=\{q\}$, which shows that there is~$T\in \mathscr{F}(S)$ with $K=p(T)$. Therefore, we can take~$C_{q}\in \mu(q)$ with~$C_{q}\subseteq T$. Now we only need to show that~ $C_{q}\not\preccurlyeq \mu(r)$. Indeed, suppose not, then there is~$C_{r}\in \mu(r)$ with~$C_{r}\subseteq C_{q}\subseteq T$. Thus~$r\in P(T)=K$; however, $r\notin K$, this is a contradiction.

\smallskip
(2) $\Longrightarrow$ (1). Assume (2) is true. We need to prove that $(Q,\mathscr{K})$ is a $T_{1}$-knowledge structure by (2) of Lemma~\ref{1-1}. Take any $q,r \in Q$ with $q\neq r$. Hence $\mu(q)\not\preccurlyeq \mu(r)$, that is, we can choose $C_{q}\in \mu(q)$ with $C_{q}\not\preccurlyeq \mu(r)$. Obviously,~$q\in p(C_{q})\in \mathscr{K}$. If~$r\in p(C_{q})$, then there is~$C_{r}\in \mu(r)$ with ~$C_{r}\subseteq C_{q}$, which is a contradiction with $C_{q}\not\preccurlyeq \mu(r)$. Thus $r\notin p(C_{q})$. Similarly, we can take~$C'_{r}\in \mu(r)$ such that~$r\in p(C'_{r})\in \mathscr{K}$ and~$q\notin p(C'_{r})$. Therefore,~$(Q,\mathscr{K})$ is a~$T_{1}$-knowledge structure.
\end{proof}

\vspace{3mm}
The following two examples show that the conditions~``$(Q,\mathscr{K})$ is bi-discriminative'' and~``for all~$q,r\in Q$ with $q\neq r$, $\mu(q)\not\subseteq \mu(r)$ and~$\mu(q)\not\subseteq \mu(r)$'' are not equivalent.

\begin{ex}\label{1-4}
Let~$Q=\{a,b\}$ and~$S=\{s_{1},s_{2}\}$. Put~$\mu(a)=\{\{\frac{0.2}{s_{1}}\}\}$ and~ $\mu(b)=\{\{\frac{0.3}{s_{1}},\frac{0.7}{s_{2}}\}\}$. Then~$(Q,S,\mu)$ is a fuzzy skill function. The delineated knowledge structure
$$\mathscr{K}=\{\emptyset, a, Q\}.$$
is not bi-discriminative since~$\mu(b)\preccurlyeq \mu(a)$. However,~$\mu(a)\not\subseteq \mu(b)$ and $\mu(b)\not\subseteq \mu(a)$.
\end{ex}

There exists a disjunctive fuzzy skill multimap~$(Q,S,\mu)$ satisfies that, for any~$q, r\in Q$ with $q\neq r$, $\mu(q)\not\subseteq \mu(r)$ and~$\mu(q)\not\subseteq \mu(r)$; however, the delineated knowledge structure is not bi-discriminative.

\begin{ex}
Assume~$Q=\{q,r\}$ and~$S=\{s_{1},s_{2}\}$. Put~$\mu(q)=\{\{\frac{0.2}{s_{1}}\}\}$ and~ $\mu(r)=\{\{\frac{0.1}{s_{1}}\},\{\frac{0.7}{s_{2}}\}\}$. Then~$(Q,S,\mu)$ is a disjunctive fuzzy skill multimap. Then the delineated knowledge structure
$$\mathscr{K}=\{\emptyset, r, Q\}$$
is not bi-discriminative since~$\mu(q)\preccurlyeq \mu(r)$. However,~$\mu(q)\not\subseteq \mu(r)$ and~$\mu(r)\not\subseteq \mu(q)$ for~$q, r\in Q$.
\end{ex}

\begin{rem}
Assume $(Q,S,\mu)$ and~$(Q,\mathscr{K})$ are described as in Example~\ref{1-4}. Then~$M_{a}=\{\{\frac{0.2}{s_{1}}\}\}$ and~$M_{b}=\{\{\frac{0.3}{s_{1}},\frac{0.7}{s_{2}}\}\}$. It shows that the condition~``$(Q,\mathscr{K})$ is bi-discriminative'' in Corollary~\ref{ccc} is not equivalent to the condition~``for any $q, r \in Q$ with~$q\neq r$, $M_{q}\notin \mu(r)$ and~$M_{r}\notin \mu(q)$''.
\end{rem}

\subsection{Distribution and merge: fuzzy skill multimap}

Similar to the applications discussed in~\cite[section 4]{ref4}, we give some similar applications of Theorems~\ref{fff} and~\ref{1-3} respectively in this section.

\begin{defn}
Assume $(Q',S',\mu')$ and~$(Q,S,\mu)$ are two fuzzy skill multimaps. The fuzzy skill multimap $(Q,S,\mu)$ is said to be prolonged~$(Q',S',\mu')$ if the following conditions hold.

\smallskip
(1)~$Q'=Q$;

\smallskip
(2)~$S'=\{s_{i}\}_{i\in I'}$ and~$S=\{s_{i}\}_{i\in I}$, where~$I\subseteq I'$;

\smallskip
(3)~$\mu(q)=\{C'\cap \{\frac{1}{s_{i}}, \frac{0}{s_{j}}\}_{i\in I, j\in I'\backslash I} | C'\in \mu'(q)\}$.
\end{defn}

\begin{prop}\label{1-6}
Assume that $(Q',\mathscr{K}')$ and~$(Q,\mathscr{K})$ are knowledge structures which are delineated by the fuzzy skill multimaps~$(Q', S', \mu')$ and~$(Q, S, \mu)$ respectively, and assume that $(Q',S',\mu')$ prolongs~$(Q,S,\mu)$. Then~$(Q', \mathscr{K}')$ is discriminative (resp. bi-discriminative) when~$(Q, \mathscr{K})$ is discriminative (resp. bi-discriminative).
\end{prop}

\begin{proof}
Let $(Q,\mathscr{K})$ be bi-discriminative. Take any~$q, r \in Q$ with~$q\neq r$. By Theorem~\ref{1-3},~$\mu(q)\not\preccurlyeq \mu(r)$, which shows that there exists~$C\in \mu(q)$ with $C\not\preccurlyeq \mu(r)$. Further, there exists~$C'\in \mu'(q)$ with $C=C'\cap \{\frac{1}{s_{i}}, \frac{0}{s_{j}}\}_{i\in I, j\in I'\backslash I} $. It follows from $C\not\preccurlyeq \mu(r)$ that~$C'\not\preccurlyeq \mu'(r)$, hence $\mu'(q)\not\preccurlyeq \mu'(r)$. Similarly, we have $\mu'(r)\not\preccurlyeq \mu'(q)$. Therefore, $(Q',\mathscr{K}')$ is bi-discriminative by Theorem~\ref{1-3}.
\end{proof}

\vspace{3mm}
Proposition~\ref{1-6} gives a similar result to~\cite[Proposition 3]{ref4}.

\begin{ex}
Assume $Q'=Q=\{q, r\}$,~$S'=\{s_{1}, s_{2}, s_{3}\}$ and~$S=\{s_{1}\}$. Put~$\mu'(q)=\{\{\frac{0.2}{s_{1}},\frac{0.3}{s_{2}}\}\}$, ~$\mu'(r)=\{\{\frac{0.2}{s_{1}},\frac{0.5}{s_{3}}\}\}$ and~$\mu(q)=\mu(r)=\{\{\frac{0.2}{s_{1}}\}\}$. Then the fuzzy skill multimaps~$(Q',S',\mu')$ prolongs the fuzzy skill multimap~$(Q,S,\mu)$. Suppose that $(Q',\mathscr{K}')$ and~$(Q,\mathscr{K})$ are knowledge structures delineated by the fuzzy skill multimaps ~$(Q',S',\mu')$ and~$(Q,S,\mu)$ respectively. Then
$$\mathscr{K}'=\{\emptyset, q, r, Q\}\ \mbox{and}\ \mathscr{K}=\{\emptyset, Q\}.$$ Then~$(Q',\mathscr{K}')$ is bi-discriminative since~$\mu(q)\not\preccurlyeq \mu(r)$ and~$\mu(r)\not\preccurlyeq \mu(q)$. However,~$(Q,\mathscr{K})$ is not discriminative. It follows that the inverse of Proposition~\ref{1-6} does not hold.
\end{ex}

By shortening a fuzzy skill multimap, we have the concept of a submultimap of a fuzzy skill multimap and a delineated substructure of a knowledge structure.

\begin{defn}
Assume $(Q',S',\mu')$ and~$(Q,S,\mu)$ are two fuzzy skill multimaps. The fuzzy skill multimap $(Q, S, \mu)$ is said to be a submultimap of $(Q',S',\mu')$ if the following conditions hold.

\smallskip
(1)~$Q\supseteq Q'$ and $S=S'=\{s_{i}\}_{i\in I}$;

\smallskip
(2) for each~$q\in Q$, we have $\mu(q)=\mu'(q)$.
\end{defn}

\begin{lemma}\label{1-7}
Assume that $(Q', S', \mu')$ and~$(Q, S, \mu)$ are two fuzzy skill multimaps, and assume that $(Q', \mathscr{K}')$ and~$(Q, \mathscr{K})$ are knowledge structures delineated by $(Q', S', \mu')$ and~$(Q, S, \mu)$ respectively. Then~$(Q,\mathscr{K})$ is a substructure of~$(Q',\mathscr{K}')$ if~$(Q,S,\mu)$ is a submultimap of~$(Q',S',\mu')$.
\end{lemma}

\begin{proof}
Assume $p$ and~$p'$ are two problem functions induced by~$(Q,S,\mu)$ and~$(Q',S',\mu')$, respectively. Take any $K\in \mathscr{K}$; then there exists~$T\in \mathscr{F}(S)$ with~$K=p(T)$. We conclude that $p(T)=p'(T)\cap Q$. For any $q\in p(T)$, it follows that~$q\in Q$ and there exists~$C_{q}\in \mu(q)=\mu'(q)$ with $C_{q}\subseteq T$. Thus~$q\in p'(T)\cap Q$. This prove~$p(T)\subseteq p'(T)\cap Q$. Conversely, take any $q'\in p'(T)\cap Q$, then~$q'\in Q$ and there exists~$C_{q'}\in \mu'(q')=\mu(q')$ with $C_{q'}\subseteq T$. Hence~$q'\in p(T)\cap Q$. This prove~$p'(T)\cap Q\subseteq p(T)$. Consequently,~$p(T)=p'(T)\cap Q$.
\end{proof}

\vspace{3mm}
From Lemma~\ref{1-7} and Proposition~\ref{dd}, we have the following result immediately.

\begin{prop}
Assume that $(Q',\mathscr{K}')$ and~$(Q,\mathscr{K})$ are knowledge structures delineated by the fuzzy skill multimaps~$(Q',S',\mu')$ and~$(Q,S,\mu)$ respectively, and assume that $(Q,S,\mu)$ is a submultimap of~$(Q',S',\mu')$. Then~$(Q,\mathscr{K})$ is discriminative (resp. bi-discriminative) if~$(Q',\mathscr{K}')$ is discriminative (resp. bi-discriminative).
\end{prop}

Heller and Repitsch defined the merge~$(Q,S,\mu)$ of a family~$\{(Q_{i},S_{i},\mu_{i})| i\in I\}$ of skill functions~\cite{ref22} that is applied to fuzzy skill functions in this section.

\begin{defn}[\cite{ref22}]\label{bbb}
{\rm Let $\{(Q_{i}, \mathscr{K}_{i})|i\in I\}$ be a family of knowledge structures. We say that a knowledge structure~$(Q,\mathscr{K})$ is a {\it mesh} of $\{(Q_{i}, \mathscr{K}_{i})|i\in I\}$ if the following conditions hold.

 \smallskip
(1)~$Q=\bigcup_{i\in I} Q_{i}$.

\smallskip
(2)~$\mathscr{K}_{i}=\mathscr{K}|_{Q_{i}}$ for all~$i\in I$.

Moreover, we say that $\{(Q_{i},\mathscr{K}_{i})|i\in I\}$ is {\it meshable}.}
\end{defn}

\begin{defn}~\label{aaa}
{\rm Assume that $\{(Q_{i},S_{i},\mu_{i})| i\in I\}$ is a family of fuzzy skill functions, and assume that $J=\bigcup_{i\in I}J_{i}$. We say that $(Q, S, \mu)$ is {\it merge} of $\{(Q_{i},S_{i},\mu_{i})| i\in I\}$ if the following conditions hold.

 \smallskip
(1)~$Q=\bigcup_{i\in I} Q_{i}$.

 \smallskip
(2)~$S=\bigcup_{i\in I} S_{i}$.

 \smallskip
(3)~$\mu(q)=\bigcup_{i\in I} \mu^{*}_{i}(q)$ for each~$q\in Q$, where~$\mu^{*}_{i}(q)=\{\{\frac{C_{i}(s_{j})}{s_{j}},\frac{0}{s_{k}}\}_{j\in J_{i}, k\in J\backslash J_{i}}| C_{i}\in \mu_{i}(q)\}$ if~$q\in Q_{i}$, and~$\mu^{*}_{i}(q)=\emptyset$ otherwise.

Moreover, if the merge~$(Q,S,\mu)$ is a fuzzy skill function, then it is said to be a {\it distributed fuzzy skill function}.}
\end{defn}

\begin{Notation}~\label{ffff}
Assume that $\{(Q_{i},S_{i},\mu_{i})| i\in I\}$ are a family of fuzzy skill functions, assume that each $\mu^{*}_{i}$ is described as in Definition~\ref{aaa}, and assume that~$(Q,S,\mu)$ is the corresponding distributed fuzzy skill function in the rest of this paper.
Additionally, let~$(Q,\mathscr{K})$ and~$\{(Q_{i},\mathscr{K}_{i})|i\in I\} (i\in I)$ be knowledge structure delineated by~$(Q,S,\mu)$ and~$(Q_{i},S_{i},\mu_{i})$ respectively.
\end{Notation}

The following results are generalizations of some of the results applied in skill functions provided by Heller and Repitsch~\cite[Section 5]{ref22}. Likewise, in this section we are always assume that a collection of fuzzy skill functions~$(Q_{i},S_{i},\mu_{i}), i\in I$ gives rise to a distributed fuzzy skill function~$(Q,S,\mu)$. Moreover, the problem functions~$p_{i}$ and~$p$ are induced by~$\mu_{i}$ and~$\mu$ respectively. From Definitions~\ref{aaaa} and~\ref{aaa}, it is obvious that the following lemma holds.

\begin{lemma}\label{tt}
For each~$i\in I$, the following statements hold.

\smallskip
(1) for any $q\in Q_{i}$ and $C_{i}\in\mu^{*}_{i}(q)$, there is~$C\in\mu(q)$ with $$C_{i}=C\cap\{\frac{1}{s_{j}}, \frac{0}{s_{k}}\}_{j\in J_{i}, k\in J\backslash J_{i}};$$

\smallskip
(2) for any $q\in Q_{i}$ and ~$C\in\mu(q)$, we have $C\cap\{\frac{1}{s_{j}}, \frac{0}{s_{k}}\}_{j\in J_{i}, k\in J\backslash J_{i}}\subseteq C$;

\smallskip
(3)~$p^{*}_{i}(T\cap \{\frac{1}{s_{j}}, \frac{0}{s_{k}}\}_{j\in J_{i}, k\in J\backslash J_{i}})\subseteq p(T\cap \{\frac{1}{s_{j}}, \frac{0}{s_{k}}\}_{j\in J_{i}, k\in J\backslash J_{i}})\cap Q_{i}\subseteq p(T)\cap Q_{i}$;

\smallskip
Further, we have

\smallskip
(4)~$p(T)=\bigcup_{i\in I}p^{*}_{i}(T\bigcap \{\frac{1}{s_{j}}, \frac{0}{s_{k}}\}_{j\in J_{i}, k\in J\backslash J_{i}})$.
\end{lemma}

\begin{prop}\label{gg}
For each~$i\in I$ the following statements are equivalent.

\smallskip
(1) For any $q\in Q$ and $C\in\mu(q)$, if $C\subseteq \{\frac{1}{s_{j}}, \frac{0}{s_{k}}\}_{j\in J_{i}, k\in J\backslash J_{i}}$ for $i\in I$ then we have $C\in \mu^{*}_{i}(q)$.

\smallskip
(2) For any $i\in I$, we have $$p^{*}_{i}(T\cap \{\frac{1}{s_{j}}, \frac{0}{s_{k}}\}_{j\in J_{i}, k\in J\backslash J_{i}})=p(T\cap \{\frac{1}{s_{j}}, \frac{0}{s_{k}}\}_{j\in J_{i}, k\in J\backslash J_{i}})\cap Q_{i}$$ for each fuzzy set~$T\in \mathscr{F}(S)$.

\smallskip
Further, (1) and (2) imply~$\mathscr{K}_{i}\subseteq \mathscr{K}|_{Q_{i}}$.
\end{prop}

\begin{proof}
(1) $\Rightarrow$ (2). Fix any $i\in I$, and take any~$T\in\mathscr{F}(S)$. By Lemma~\ref{tt}, it suffices to prove that $$p(T\cap \{\frac{1}{s_{j}}, \frac{0}{s_{k}}\}_{j\in J_{i}, k\in J\backslash J_{i}})\cap Q_{i}\subseteq p^{*}_{i}(T\cap \{\frac{1}{s_{j}}, \frac{0}{s_{k}}\}_{j\in J_{i}, k\in J\backslash J_{i}}).$$  Pick any~$q\in p(T\cap \{\frac{1}{s_{j}}, \frac{0}{s_{k}}\}_{j\in J_{i}, k\in J\backslash J_{i}})\cap Q_{i}$. Then there is~$C\in\mu(q)$ with $C\subseteq T\cap\{\frac{1}{s_{j}}, \frac{0}{s_{k}}\}_{j\in J_{i}, k\in J\backslash J_{i}}$, hence $C\subseteq \{\frac{1}{s_{j}}, \frac{0}{s_{k}}\}_{j\in J_{i}, k\in J\backslash J_{i}}$. By our assumption, we have $C\in \mu^{*}_{i}(q)$, hence it follows that $q\in p^{*}_{i}(T\cap \{\frac{1}{s_{j}}, \frac{0}{s_{k}}\}_{j\in J_{i}, k\in J\backslash J_{i}})$.

(2) $\Rightarrow$ (1). Take any $q\in Q$ and $C\in\mu(q)$. Assume that $C\subseteq \{\frac{1}{s_{j}}, \frac{0}{s_{k}}\}_{j\in J_{i}, k\in J\backslash J_{i}}$ for $i\in I$. Since $C\in\mu(q)$ and $q\in Q_{i}$, it follows that $$q\in p(C)\cap Q_{i}=p(C\cap \{\frac{1}{s_{j}}, \frac{0}{s_{k}}\}_{j\in J_{i}, k\in J\backslash J_{i}})\cap Q_{i}=p^{*}_{i}(C\cap \{\frac{1}{s_{j}}, \frac{0}{s_{k}}\}_{j\in J_{i}, k\in J\backslash J_{i}}).$$ Thus, there exists $C_{i}\in \mu^{*}_{i}(q)$ such that $C_{i}\subseteq C\cap \{\frac{1}{s_{j}}, \frac{0}{s_{k}}\}_{j\in J_{i}, k\in J\backslash J_{i}}=C$. However, since $p$ is a skill function, we have that $C_{i}= C\cap \{\frac{1}{s_{j}}, \frac{0}{s_{k}}\}_{j\in J_{i}, k\in J\backslash J_{i}}\in\mu^{*}_{i}(q)$, thus $C_{i}=C$.

In order to prove that (2) implies~$\mathscr{K}_{i}\subseteq \mathscr{K}|_{Q_{i}}$, take any state~$K_{i}$ in~$\mathscr{K}_{i}$. Hence there exists~$T\in \mathscr{F}(S)$ such that$$K_{i}=p^{*}_{i}(T\cap \{\frac{1}{s_{j}}, \frac{0}{s_{k}}\}_{j\in J_{i}, k\in J\backslash J_{i}})=p(T\cap \{\frac{1}{s_{j}}, \frac{0}{s_{k}}\}_{j\in J_{i}, k\in J\backslash J_{i}})\cap Q_{i}.$$ Since~$p(T\cap \{\frac{1}{s_{j}}, \frac{0}{s_{k}}\}_{j\in J_{i}, k\in J\backslash J_{i}})\in \mathscr{K}$, we conclude that $K_{i}\in\mathscr{K}|_{Q_{i}}$.
\end{proof}

\begin{prop}\label{ggg}
For each~$i\in I$ the following statements are equivalent.

\smallskip
(1) For any~$q\in Q$ and $C\in\mu(q)$, if $q\in Q_{i}$ for $i\in I$ then $C\subseteq \{\frac{1}{s_{j}}, \frac{0}{s_{k}}\}_{j\in J_{i}, k\in J\backslash J_{i}}$.

\smallskip
(2) For each $i\in I$, we have $p(T\cap \{\frac{1}{s_{j}}, \frac{0}{s_{k}}\}_{j\in J_{i}, k\in J\backslash J_{i}})\cap Q_{i}=p(T)\cap Q_{i}$ for each fuzzy set~$T\in \mathscr{F}(S)$.

\smallskip
Further, (1) and (2) imply~$\mathscr{K}_{i}=\mathscr{K}|_{Q_{i}}$.
\end{prop}

\begin{proof}
(1) $\Rightarrow$ (2). Take any $T\in \mathscr{F}(S)$. By Lemma~\ref{tt}, it suffices to prove $$p(T)\cap Q_{i}\subseteq p(T\cap \{\frac{1}{s_{j}}, \frac{0}{s_{k}}\}_{j\in J_{i}, k\in J\backslash J_{i}})\cap Q_{i}.$$ Let $q\in p(T)\cap Q_{i}$ for some $i\in I$. Hence there exists~$C\in \mu(q)$ with~$C\subseteq T$. By our assumption, we have $C\subseteq \{\frac{1}{s_{j}}, \frac{0}{s_{k}}\}_{j\in J_{i}, k\in J\backslash J_{i}}$, hence $C\subseteq T\cap \{\frac{1}{s_{j}}, \frac{0}{s_{k}}\}_{j\in J_{i}, k\in J\backslash J_{i}}$. Consequently, we have~$q\in p(T\cap \{\frac{1}{s_{j}}, \frac{0}{s_{k}}\}_{j\in J_{i}, k\in J\backslash J_{i}})\cap Q_{i}$.

(2) $\Rightarrow$ (1).  Suppose that there exists~$q\in Q_{i}$ for some $i\in I$ and $C\in \mu(q)$ such that $C\not\subseteq \{\frac{1}{s_{j}}, \frac{0}{s_{k}}\}_{j\in J_{i}, k\in J\backslash J_{i}}$. However, $q\in p(C)\cap Q_{i}$, then it follows from our assumption that~$q\in p(C\cap \{\frac{1}{s_{j}}, \frac{0}{s_{k}}\}_{j\in J_{i}, k\in J\backslash J_{i}})\cap Q_{i}$. Therefore, there exists~$C'\in \mu(q)$ with~$C'\subseteq C\cap \{\frac{1}{s_{j}}, \frac{0}{s_{k}}\}_{j\in J_{i}, k\in J\backslash J_{i}}\subset C$, which contradicts the minimality.

 To show that~$\mathscr{K}_{i}=\mathscr{K}|_{Q_{i}}$ is implied by (2) we only need to observe that for any~$i\in I$
$$\mathscr{K}_{i}=\{p^{*}_{i}(T\cap \{\frac{1}{s_{j}}, \frac{0}{s_{k}}\}_{j\in J_{i}, k\in J\backslash J_{i}})| T\in \mathscr{F}(S)\}=\{p(T)\cap Q_{i}|T\in \mathscr{F}(S)\}=\mathscr{K}|_{Q_{i}}.$$
\end{proof}

The knowledge structure~$(Q,\mathscr{K})$, which is delineated by the distributed fuzzy skill function~$\mu$, is a mesh of the family of the knowledge structures~$\{(Q_{i},\mathscr{K}_{i}): i\in I\}$ delineated by the component fuzzy skill functions~$\mu_{i} (i\in I)$ under the condition such that if the domains of any fuzzy skill function overlap, then they all assign the same competencies to the common problems.

if the problem domains are pairwise disjoint, the we have the following corollary from Propositions~\ref{gg} and~\ref{ggg}.

\begin{cor}\label{1-8}
Assume $(Q, S, \mu)$ is the corresponding distributed fuzzy skill function of the family of fuzzy skill functions $\{(Q_{i},S_{i},\mu_{i}): i\in I\}$.

\smallskip
(1) If $\{Q_{i}: i\in I\}$ is pairwise disjoint then the knowledge structure~$\mathscr{K}$ delineated by~$\mu$ is a mesh of the component knowledge structures~$\{\mathscr{K}_{i}: i\in I\}$, that is, ~$\mathscr{K}_{i}=\mathscr{K}|_{Q_{i}}$ for each~$i\in I$.

\smallskip
(2) If the skill domains~$\{S_{i}: i\in I\}$ is pairwise disjoint then~$\mathscr{K}_{i}\subseteq \mathscr{K}|_{Q_{i}}$ for each~$i\in I$.
\end{cor}

\begin{defn}\label{1-9}
{\rm We say that $\{(Q_{i},S_{i},\mu_{i})| i\in I\}$ is {\it consistent} if the family~$\{(Q_{i},\mathscr{K}_{i})|i\in I\}$ is meshable.
Further,~$\{(Q_{i},S_{i},\mu_{i})| i\in I\}$ is~{\it $(Q,S,\mu)$-consistent} if~$(Q,\mathscr{K})$ is mesh of~$\{(Q_{i},\mathscr{K}_{i})|i\in I\}$.}
\end{defn}

Heller et al. gave the definition of consistent if~$\{(Q_{i},S_{i},\mu_{i})| i\in I\}$ is a family of skill functions~\cite{ref23}, which applied to fuzzy skill functions in this section. Hence, the following lemma holds by Corollary~\ref{1-8} and Definition~\ref{1-9}.

\begin{lemma}~\label{cccc}
If~$\{Q_{i}|i\in I\}$ is pairwise disjoint, then~$\{(Q_{i},S_{i},\mu_{i})| i\in I\}$ is~$(Q,S,\mu)$-consistent.
\end{lemma}

\begin{lemma}~\label{bb}
Let $\{Q_{i}|i\in I\}$ be pairwise disjoint. Then~$(Q_{i},\mathscr{K}_{i})$ is discriminative (resp. bi-discriminative) for each~$i\in I$ if~$(Q,\mathscr{K})$ is discriminative (resp. bi-discriminative).
\end{lemma}

\begin{proof}
Let $(Q,\mathscr{K})$ be discriminative. Since~$\{Q_{i}|
i\in I\}$ is pairwise disjoint, it follows from Lemma~\ref{cccc} that~$\{(Q_{i},S_{i},\mu_{i})| i\in I\}$ is~$(Q,S,\mu)$-consistent. Hence,~$(Q,\mathscr{K})$ is a mesh of~$\{(Q_{i},\mathscr{K}_{i})|i\in I\}$. From Definition~\ref{bbb}, it clear that~$(Q_{i},\mathscr{K}_{i})$ is a substructure of~$(Q,\mathscr{K})$ for any $i\in I$. From Proposition~\ref{dd}, $(Q_{i},\mathscr{K}_{i})$ is discriminative for any~$i\in I$.
\end{proof}

\begin{lemma}~\label{dddd}
Let $\{S_{i}|i\in I\}$ be pairwise disjoint. Then~$(Q,\mathscr{K})$ is discriminative (resp. bi-discriminative) if~$(Q_{i},\mathscr{K}_{i})$ is discriminative (resp. bi-discriminative) for each~$i\in I$.
\end{lemma}

\begin{proof}
Clearly, we only need to consider the case of discrimination. Suppose that~$(Q_{i},\mathscr{K}_{i})$ is discriminative for all~$i\in I$. Take any~$q,r\in Q$ with~$q\neq r$. We divide the proof into the following two cases.

\smallskip
{\bf Case 1}: $q, r\in Q_{j}$ for some~$j\in I$.

\smallskip
Since~$(Q_{j},\mathscr{K}_{j})$ is discriminative, it follows from Theorem~\ref{fff} that, without loss of generality, there exists~$C_{q}\in \mu^{*}_{j}(q)\subseteq \mu(q)$ with~$C_{q}\not\preccurlyeq \mu^{*}_{j}(r)$, then for any~$C_{r}\in \mu^{*}_{j}(r)$ there exists~$s_{C_{r}}\in S$ with $C_{q}(s_{C_{r}})<C_{r}(s_{C_{r}})$. Since~$\{S^{*}_{i}|i\in I\}$ is pairwise disjoint, we conclude that $C_{q}\cap (\bigcup_{i\in I\backslash \{j\}}\{\frac{1}{s_{k}}, \frac{0}{s_{l}}\}_{k\in J_{i}, l\in J\backslash J_{i}})=\emptyset$, hence $C_{q} \cap (\bigcup_{i\in I\backslash \{j\}}(\bigcup \mu^{*}_{i}(r)))=\emptyset$. Therefore, for each~$C_{r}\in \bigcup_{i\in I}(\bigcup\mu^{*}_{i}(r))=\mu(r)$, there exists~$s_{C_{r}}\in S$ with $C_{q}(s_{C_{r}})<C_{r}(s_{C_{r}})$. Then it follows that $C_{q}\not\preccurlyeq \mu(r)$, hence $\mu(q)\not\preccurlyeq \mu(r)$.

\smallskip
{\bf Case 2}:  $q\in Q_{j}$ and~$r\notin Q_{j}$ for some~$j\in I$.

\smallskip
Take any $C_{q}\in \mu^{*}_{j}(q)\subseteq \mu(q)$. Because the family $\{S_{i}|i\in I\}$ is pairwise disjoint, it follows that $C_{q}\cap (\bigcup_{i\in I\backslash \{j\}}\{\frac{1}{s_{k}}, \frac{0}{s_{l}}\}_{k\in J_{i}, l\in J\backslash J_{i}})=\emptyset$. From $\mu^{*}_{j}(r)=\emptyset$, we have $C_{q}\cap (\bigcup_{i\in I}(\bigcup \mu^{*}_{i}(r)))=\emptyset$. From $\bigcup_{i\in I}(\bigcup \mu^{*}_{i}(r))=\mu(r)$, we conclude that $\mu(q)\not\preccurlyeq \mu(r)$.

From Theorem~\ref{fff}, $(Q,\mathscr{K})$ is discriminative.
\end{proof}

The following theorem is immediate from Lemmas~\ref{bb} and~\ref{dddd}, which is a precisely answer to Question~\ref{kk} with regard to the separability (resp. bi-separability) of items.

\begin{thm}~\label{eeee}
Let~$\{Q_{i}|i\in I\}$ and~$\{S_{i}|i\in I\}$ be two pairwise disjoint family. Then the following are equivalent.

\smallskip
(1)~$(Q,\mathscr{K})$ is discriminative (resp. bi-discriminative).

\smallskip
(2)~each $(Q_{i}, \mathscr{K}_{i})$ is discriminative (resp. bi-discriminative).
\end{thm}

The following two examples show that the reverse of Lemmas~\ref{bb} and~\ref{dddd} do not hold respectively.

\begin{ex}
Let~$(Q_{i},S_{i},\mu_{i})$ for each~$i\in\{1,2\}$ and~$(Q,S,\mu)$ be described as follows.
\begin{equation*}
\begin{aligned}
&Q_{1}=\{a,b\}, S_{1}=\{s_{1},s_{2},s_{3},s_{4}\}, \\
&\mu_{1}(a)=\{\{\frac{0.1}{s_{1}}, \frac{0.7}{s_{2}}\}, \{\frac{0.4}{s_{2}}, \frac{0.6}{s_{3}}\}\},\\
&\mu_{1}(b)=\{\{\frac{0.2}{s_{1}}, \frac{0.5}{s_{3}}\}, \{\frac{0.5}{s_{3}}, \frac{0.5}{s_{4}}\}\};\\
&Q_{2}=\{c,d\}, S_{2}=\{s_{1},s_{3},s_{4},s_{5}\},\\
&\mu_{2}(c)=\{\{\frac{0.2}{s_{1}}, \frac{0.5}{s_{3}}\}, \{\frac{0.5}{s_{3}}, \frac{0.5}{s_{4}}\}\},\\
&\mu_{2}(d)=\{\{\frac{0.2}{s_{1}}, \frac{0.5}{s_{4}}\}, \{\frac{0.5}{s_{3}}, \frac{0.5}{s_{5}}\}\};\\
&Q=Q_{1}\cup Q_{2}, S=S_{1}\cup S_{2}, \mu(a)=\mu^{*}_{1}(a), \mu(b)=\mu^{*}_{1}(b),\\
&\mu(c)=\mu^{*}_{2}(c), \mu(d)=\mu^{*}_{2}(d).
\end{aligned}
\end{equation*}

They delineate the knowledge structures
\begin{equation*}
\begin{aligned}
&\mathscr{K}_{1}=\{\emptyset, \{a\}, \{b\}, \{a, b\}\};\\
&\mathscr{K}_{2}=\{\emptyset, \{c\}, \{d\}, \{c, d\}\};\\
&\mathscr{K}=\{\emptyset, \{a\}, \{d\}, \{b, c\}, \{a, d\}, \{a, b, c\}, Q\}.
\end{aligned}
\end{equation*}

Clearly,~$(Q,S,\mu)$ is the distributed fuzzy skill function of~$\{(Q_{i},S_{i},\mu_{i})|i\in\{1,2\}\}$ and~$Q_{1}\cap Q_{2}=\emptyset$. Clearly,~$(Q_{1},\mathscr{K}_{1})$ and~$(Q_{2},\mathscr{K}_{2})$ delineated by~$(Q_{1},S_{1},\mu_{1})$ and~$(Q_{2},S_{2},\mu_{2})$ respectively are both bi-discriminative. However, the knowledge structure~$(Q,\mathscr{K})$ delineated by~$(Q,S,\mu)$ is not discriminative since~$\mathscr{K}_{b}=\mathscr{K}_{c}$. From another perspective,~$(Q_{1},\mathscr{K}_{1})$ is bi-discriminative since~$\mu_{1}(a)\not\preccurlyeq \mu_{1}(b)$ and~$\mu_{1}(b)\not\preccurlyeq \mu_{1}(a)$ by Theorem~\ref{1-3}. Moreover, $(Q_{2},\mathscr{K}_{2})$ is bi-discriminative in the same way. However,~$(Q,\mathscr{K})$ is not discriminative since $\mu^{*}(b)=\mu^{*}(c)$.
\end{ex}

\begin{ex}
Let~$(Q_{i},S_{i},\mu_{i})$ for each $i\in\{1,2,3\}$ and~$(Q,S,\mu)$ be described as follows.
\begin{equation*}
\begin{aligned}
&Q_{1}=\{a,b\}, S_{1}=\{s_{1}\}, \mu_{1}(a)=\mu_{1}(b)=\{\{\frac{0.2}{s_{1}}\}\};\\
&Q_{2}=\{b,c\}, S_{2}=\{s_{2}\}, \mu_{2}(b)=\mu_{2}(c)=\{\{\frac{0.4}{s_{2}}\}\};\\
&Q_{3}=\{a,c\}, S_{3}=\{s_{3}\}, \mu_{3}(a)=\mu_{3}(c)=\{\{\frac{0.6}{s_{3}}\}\};\\
&Q=Q_{1}\bigcup Q_{2}\bigcup Q_{3}, S=S_{1}\bigcup S_{2}\bigcup S_{3}, \mu^{*}(a)=\{\{\frac{0.2}{s_{1}}\}, \{\frac{0.6}{s_{3}}\}\},\\ &\mu^{*}(b)=\{\{\frac{0.2}{s_{1}}\}, \{\frac{0.4}{s_{2}}\}\},
\mu^{*}(c)=\{\{\frac{0.4}{s_{2}}\}, \{\frac{0.6}{s_{3}}\}\}.
\end{aligned}
\end{equation*}

Clearly,~$(Q,S,\mu)$ is the distributed fuzzy skill function of~$\{(Q_{i},S_{i},\mu_{i})|i\in\{1,2,3\}\}$ and~$\{S_{i}|i\in\{1,2,3\}\}$ is pairwise disjoint. By Theorem~\ref{fff},~$(Q_{i},\mathscr{K}_{i})(i\in\{1,2,3\})$ is not discriminative since~$\mu_{1}(a)\not\preccurlyeq \mu_{1}(b)$ and~$\mu_{1}(b)\not\preccurlyeq \mu_{1}(a)$. However, it follows from Theorem~\ref{1-3} that $(Q,\mathscr{K})$is bi-discriminative.
\end{ex}

\section{Conclusions}

This present paper gives the discussions about the knowledge structures delineated by fuzzy multimaps. Around Question~\ref{1-5} inspired by a question in~\cite[problem 6 in Chapter 6 ]{falmagne2011learning}, section 3.1 draws characterizations of fuzzy skill multimaps, ensuring the knowledge structures delineated by fuzzy skill multimaps are knowledge space, simple closure space, and learning space respectively. Around Question~\ref{kkk}, section 3.2, which generalizes and extends the results about the separability and bi-separability~\cite{ref4}, discusses the separability and the bi-separability of items in knowledge structures delineated by fuzzy skill multimaps respectively.

Let~$(Q,\mathscr{K})$ be delineated by the distributed fuzzy skill multimaps ~$(Q,S,\mu)$. Theorem~\ref{fff}, which is a generalization of ~\cite[Theorem 1]{ref4}, shows that the separability of items in~$(Q,\mathscr{K})$ iff, for any~$q,r\in Q$ with~$q\neq r$, $\mu(q)\not\preccurlyeq \mu(r)$ or~$\mu(r)\not\preccurlyeq \mu(q)$. Considering the minimum elements of~$\mu(q)$ for each~$q\in Q$, we prove in Theorem~\ref{ttt} that the separability of items in~$(Q,\mathscr{K})$ iff, for any~$q,r\in Q$ with~$q\neq r$, either there exists $C\in\mu(q)$ with $M_{q, C}\neq M_{r, D}$ for any $D\in\mu(r)$, or there exists $D\in\mu(r)$ with $M_{r, D}\neq M_{q, C}$ for any $C\in\mu(q)$. These theorems give characterizations of fuzzy skill multimap~$(Q,S,\mu)$ such that~$(Q,\mathscr{K})$ delineated by~$(Q,S,\mu)$ is discriminative. In general, the ability to answer items can inferred from individual's level of proficiency in skills and the level of proficiency in skills required to solve items. For instance, for any competence~$C\in\mu(q)$, there exists~$C'\in\mu(r)$ with $C'\subseteq C$. If an individual can master item~$q$, then we can infer that the individual can master item~$r$ without repeated testing. Moreover, without repeated testing can reduce human behavior~\cite[Page 96]{falmagne2011learning} in testing situation for knowledge assessment, which is efficient to construct knowledge structure.

In addition, section 3.3 gives a precise answer to Question~\ref{kk} with reference to the separability (resp. bi-separability) of items in knowledge structures delineated by fuzzy skill multimaps, where Question~\ref{kk} was raised by Heller and Repitsch~\cite{ref22}. Lemma~\ref{bb} gives a sufficient condition such that substructures can inherit the separability (resp.  bi-separability) from parent knowledge structure, and Lemma~\ref{dddd} gives a sufficient condition to ensure parent knowledge structure can preserve the separability (resp. bi-separability) from substructures. In further practical experiments, if both~$\{Q_{i}|i\in I\}$ and~$\{S_{i}|i\in I\}$ are pairwise disjoint, the separability (resp.  bi-separability) of substructures is quickly inferred from parent knowledge structure, and the separability (resp.  bi-separability) of parent knowledge structure can also deduce from substructures.


\begin{thebibliography}{99}\addtolength{\itemsep}{-1.1ex}
\bibitem{falmagne2011learning}
J.C. Falmagne, J.P. Doignon. Learning spaces: Interdisciplinary applied mathematics. Berlin, Heidelberg: Springer, 2011.

\bibitem{ref2}
J.P. Doignon, J.C. Falmagne. Space for the assessment of knowledge. International Journal of Man-Machine Studies, 1985(23): 175-196.

\bibitem{ref3}
J.P. Doignon, J.C. Falmagne, E. Cosvn. Learning space: a mathematical compendium. In J. C. Falmagne, and et al. (Eds.), Knowledge spaces: Applications in education, Berlin and Heidelberg: Springer-Verlag, 2013: 135-145.

\bibitem{ref4}
G. Ge, J. Li. A note on the separability of items in knowledge structures delineated by skill multimaps, J. Math. Psych., 2020(98): 102427.

\bibitem{ref5}
W. Sun, J. Lin, X. Ge, Y.D. Lin. Knowledge structures delineated by fuzzy skill maps. Fuzzy Sets and Systems, 2021(407) (3): 50-66.

\bibitem{LCL}
F. Lin, X. Cao, J. Li. The language of pre-topology in knowledge spaces, arXiv:2111.14380.

\bibitem{ref6}
L.A. Zadeh. Fuzzy sets, Inf. Control, 1965(8)(3): 338-353.

\bibitem{ref7}
J.P. Doignon, J.C. Falmagne. Knowledge spaces: Berlin and Heldelberg and New York: Springer-Verlag, 1999.

\bibitem{ref8}
D. Eppstein, J.C. Falmagne, H.B. Uzun. On verifying and engineering the wellgradedness of a union-closed family, J. Math. Psychol., 2009(53)(1): 34-39.

\bibitem{ref9}
J.C. Falmagne, M. Koppen, M. Villano, J.P. Doignon, L. Johannesen. Introduction to knowledge spaces: how to build, test, and search thenPsychol. Rev., 1990(97)(2): 201-224.

\bibitem{ref10}
J. Heller. A formal framework for characterizing querying algorithms, J. Math. Psychol., 2004(48)(1): 1-8.

\bibitem{ref11}
M. Kambouri, M. Koppen, M. Villano, J.C. Falmagne. Knowledge assessment: tapping human expertise by the query routine, Int. J. Hun Comput. Stud., 1994(40)(1): 119-151.

\bibitem{ref12}
M. Koppen, J.P. Doignon. How to build a knowledge space by querying an expert, J. Math. Psychol., 1990(34)(3): 311-331.

\bibitem{ref13}
M. Koppen. Extracting human expertise for constructing knowledge spaces: an algorithm, J. Math. Psychol., 1993(37)(1): 1-20.

\bibitem{ref14}
C.E. Muller. A procedure for facilitating an expert's judgements on a set of rules, in: Mathematical Psychology in Progress, Springer, 1989, pp. 157-170.

\bibitem{ref15}
M. Schrepp, T. Held. A simulation study concerning the effect of errors on the establishment of knowledge spaces by querying experts, J. Math. Psychol., 1995(39)(4): 376-382.

\bibitem{ref16}
E. Cos yn,  N. Thiery. A practical procedure to build a knowledge structure, J. Math. Psychol., 2000(44)(3): 383-407.

\bibitem{ref17}
J.P. Doignon. Knowledge spaces and skill assignments, in: Contributions to Mathematical Psychology, Psychometrics, and Methodology.Springer, 1994, pp. 111-121.

\bibitem{ref18}
I. Duntsch, G. Gediga. Skills and knowledge structures, Br. J. Math. Stat. Psychol., 1995(48)(1): 9-27.

\bibitem{ref19}
G. Gediga, I. Duntsch. Skill set analysis in knowledge structures, Br. J. Math. Stat. Psychol., 2002(55)(2): 361-384.

\bibitem{ref20}
J. Heller, A. Unlu, D. Albert. Skills, competencies and knowledge structures, in: Knowledge Spaces, Springer-Verlag, New York, 2013, pp. 229-242.

\bibitem{ref21}
J. Heller, T. Augustin, C. Hockemeyer, L. Stefanutti, D. Albert. Recent developments in competence-based knowledge space theory, in: Knowledge Spaces: Applications in Education, Springer-Verlag, New York, 2013, pp. 243-286.

\bibitem{ref22}
J. Heller, C. Repitsch, (2008). Distributed skill functions and the meshing of knowledge structures, J. Math. Psychol., 2008(52): 147-157.

\bibitem{ref23}
J. Heller, T. Augustin, C. Hockemeyer, L. Stefanutti,  D. Albert. Recent developments in competence-based knowledge space theory. In J. C. Falmagne, \& et al. (Eds.), Knowledge spaces: Applications in education, Berlin and Heidelberg: Springer-Verlag, 2013, pp.243-286.

\end{thebibliography}
\end{document}